\documentclass[reqno]{amsart}

\usepackage{amssymb}

\usepackage[foot]{amsaddr}

\usepackage[cmtip,all]{xy}
\SilentMatrices 

\copyrightinfo{2013}{J. D. Christensen and G. Wang}

\newtheorem{thm}{Theorem}[section]
\newtheorem{lem}[thm]{Lemma}
\newtheorem{pro}[thm]{Proposition}
\newtheorem{cor}[thm]{Corollary}
\newtheorem{conj}[thm]{Conjecture}

\theoremstyle{definition}
\newtheorem{de}[thm]{Definition}

\theoremstyle{remark}
\newtheorem{rmk}[thm]{Remark}

\numberwithin{equation}{section}

\makeatletter
\def\subsection{\@startsection{subsection}{2}%
  \z@{.5\linespacing\@plus.7\linespacing}{.5\linespacing}%
  {\normalfont\bfseries}}
\makeatother

\let\leq\leqslant
\let\geq\geqslant

\usepackage[letterpaper,margin=1.25in]{geometry}
\usepackage{enumerate}

\newcommand{\st}{\mid}

\newcommand{\stmod}[1]{\mathsf{stmod}(#1)}
\newcommand{\StMod}[1]{\mathsf{StMod}(#1)}
\newcommand{\modu}[1]{\mathsf{mod}(#1)}
\newcommand{\Mod}[1]{\mathsf{Mod}(#1)}
\newcommand{\sfS}{\mathsf{S}}
\newcommand{\sfT}{\mathsf{T}}

\newcommand{\sHom}{\underline{\mathrm{Hom}}}
\newcommand{\Hom}{{\mathrm{Hom}}}
\newcommand{\PHom}{{\mathrm{PHom}}}

\newcommand{\smd}[1]{\langle #1 \rangle}
\newcommand{\thick}[1]{\mathsf{Thick}\langle #1 \rangle}
\newcommand{\thickC}[2]{\mathsf{Thick}_{#1}\langle #2 \rangle}
\newcommand{\loc}[1]{\mathsf{Loc}\langle #1 \rangle}
\newcommand{\locC}[2]{\mathsf{Loc}_{#1}\langle #2 \rangle}

\newcommand{\id}{\mathrm{id}}
\newcommand{\res}{\mathrm{res}}
\newcommand{\len}{\mathrm{len}}
\newcommand{\Flen}{\mathrm{len}^{\mathrm{F}}}

\newcommand{\gl}{\mathrm{gl}}
\newcommand{\sgl}{\mathrm{sgl}}
\newcommand{\stgl}{\mathrm{stgl}}
\newcommand{\rl}{\text{rad\,len}}

\newcommand{\up}{{\uparrow}}
\newcommand{\down}{{\downarrow}}
\newcommand{\iso}{\cong}
\newcommand{\lra}{\longrightarrow}
\newcommand{\llra}[1]{\stackrel{#1}{\lra}}  

\newcommand{\im}{\mathrm{im}}
\newcommand{\soc}{\mathrm{soc}}
\newcommand{\rad}{\mathrm{rad}}

\newcommand{\F}{\mathbb{F}}
\newcommand{\bP}{\mathbb{P}}
\newcommand{\bQ}{\mathbb{Q}}
\newcommand{\bS}{\mathbb{S}}

\newcommand{\Z}{\mathbb{Z}}
\newcommand{\cF}{\mathcal{F}}
\newcommand{\cG}{\mathcal{G}}
\newcommand{\cI}{\mathcal{I}}
\newcommand{\cJ}{\mathcal{J}}
\newcommand{\cP}{\mathcal{P}}

\newcommand{\cQ}{\mathcal{Q}}
\newcommand{\cS}{\mathcal{S}}

\newcommand{\stF}{\mathrm{st}\mathcal{F}}
\newcommand{\stG}{\mathrm{st}\mathcal{G}}
\newcommand{\sG}{\mathrm{s}\mathcal{G}}

\usepackage{tikz}
\tikzstyle{pre}=[<-,shorten <=2pt,shorten >=1.3pt,>=stealth,semithick]
\tikzstyle{post}=[->,shorten >=2pt,shorten <=1.3pt,>=stealth,semithick]
\tikzstyle{dot}=[circle, draw, fill=black!50, inner sep=0pt, minimum width=3pt]
\tikzstyle{X}=[above left=-2.5pt,font=\footnotesize]
\tikzstyle{Y}=[above right=-2.5pt,xshift=-1.8pt,font=\footnotesize]

\newcommand{\Dstartt}[1]{\path (#1) node[dot] {}}
\newcommand{\Dend}{;}
\newcommand{\Draw}[4]{+#3 node[dot](next) {} +(0,0) edge[#2] node[#4] {#1} (next.center) ++#3}
\newcommand{\Dnext}[3]{\Draw{#1}{#2}{#3}{#1}}
\newcommand{\DXo}[1]{\Dnext{}{post}{#1}}
\newcommand{\DYo}[1]{\Dnext{}{post}{#1}}
\newcommand{\Dxo}[1]{\Dnext{}{pre}{#1}}
\newcommand{\Dyo}[1]{\Dnext{}{pre}{#1}}
\newcommand{\DX} {\DXo{(-.7,-1)}}

\newcommand{\Dx} {\Dxo{( .7, 1)}}
\newcommand{\Dxw}{\Dxo{(  1, 1)}}
\newcommand{\DY} {\DYo{( .7,-1)}}
\newcommand{\DYw}{\DYo{(  1,-1)}}
\newcommand{\Dy} {\Dyo{(-.7, 1)}}

\newcommand{\blank}{-}

\newcommand{\dfn}[1]{\textbf{\boldmath{#1}}}

\newcommand{\xqedhere}[1]{\rlap{\hbox to#1{\hfil\llap{\ensuremath{\qed}}}}}

\usepackage{ifpdf}
\ifpdf  \usepackage[pdftex,bookmarks=false]{hyperref}
\else   \usepackage[hypertex]{hyperref}
\fi

\begin{document}

\title{Ghost numbers of Group Algebras II}

\author{J. Daniel Christensen}
\address{Department of Mathematics\\
University of Western Ontario\\
London, ON N6A 5B7, Canada}
\email{jdc@uwo.ca}

\author{Gaohong Wang}
\email{gwang72@uwo.ca}


\date{February 19, 2015}


\begin{abstract}
We study several closely related invariants of the group algebra $kG$ of a finite group.
The basic invariant is the ghost number, which measures the failure of the generating hypothesis
and involves finding non-trivial composites of maps each of which induces the zero map 
in Tate cohomology (``ghosts'').
The related invariants are the simple ghost number, which considers maps which are
stably trivial when composed with any map from a suspension of a simple module,
and the strong ghost number, which considers maps which are ghosts after restriction
to every subgroup of $G$.
We produce the first computations of the ghost number for non-$p$-groups, 
e.g., for the dihedral groups at all primes, as well as many new bounds.
We prove that there are close relationships between the three invariants,
and make computations of the new invariants for many families of groups.
\end{abstract}

\maketitle

\tableofcontents

\newpage

\section{Introduction}\label{se:intro}

In this paper, we study several closely related invariants of a group algebra $kG$,
where $G$ is a finite group, and $k$ is a field whose characteristic $p$
divides the order of $G$.
To describe these invariants, we work in the stable module
category $\StMod{kG}$, which is the triangulated category formed from
the category of $kG$-modules by killing the maps that factor through a projective.
A map $f$ in $\StMod{kG}$ is called a \dfn{ghost} if it induces the zero map
in Tate cohomology, or equivalently, if $\sHom(\Omega^i k, f) = 0$ for each $i \in \Z$.
Our most basic invariant is the \dfn{ghost number} of $kG$, which is the
smallest $n$ such that every composite of $n$ ghosts in $\thick k$ is zero.
Here $\thick k$ denotes the thick subcategory generated by the trivial module.
When there are no non-trivial ghosts in $\thick k$ (so $n = 1$), we say that the 
\dfn{generating hypothesis} holds for $kG$.
This is motivated by Freyd's generating hypothesis in stable homotopy theory~\cite{freydGH},
which is still an open question.
In a series of papers~\cite{GH for p,GH split,admit,GH per} (with a minor correction made below),
it has been shown that the generating hypothesis holds for $kG$ if and only
if the Sylow $p$-subgroup of $G$ is $C_2$ or $C_3$.
However, computing the ghost number in cases where it is larger than one has proven to be difficult.  
Some preliminary work was done in~\cite{Gh in rep}, where the ghost numbers
of cyclic $p$-groups were computed, and various upper and lower bounds were
obtained in other cases.
Substantial progress was made in our previous paper~\cite{Gh num}, where we
computed the ghost numbers of $k(C_3 \times C_3)$ and other algebras of wild representation
type, as well as the ghost numbers of dihedral $2$-groups, the first non-abelian computations.

In the present paper, we extend the past work in two different ways.
Our initial motivation was to produce the first computations of ghost numbers for non-$p$-groups.
For a $p$-group, $\thick k$ coincides with $\stmod{kG}$, the full subcategory
of finitely generated modules, which allows one to use induction from a subgroup to produce
ghosts in $\thick k$.
But for a general $p$-group, $\thick k$ is usually a proper subcategory of $\stmod{kG}$,
which makes things more delicate.
Nevertheless, we obtain a variety of exact computations of ghost numbers in this
setting, e.g., for all dihedral groups at all primes, as well as new bounds.
One of our new techniques is to produce ghosts for $kG$ by inducing
up a ghost from a subgroup and then projecting onto the principal block.
We show that this composite is faithful, 
and so when $\thick k$ coincides with the principal block of $\stmod{kG}$,
we are able to use this technique to study the ghost number of $kG$.
As an example, we prove that the ghost number is finite in this situation.
Our main results on ghost numbers are described in the detailed summary below.

Our work on non-$p$-groups led us to realize the importance of another
invariant in this setting, which is the \dfn{simple ghost number}, a
concept suggested in~\cite{GH general}.
A \dfn{simple ghost} is a map $f$ such that $\sHom(\Omega^i S, f) = 0$ for each
simple module $S$ and each $i \in \Z$, 
and the \dfn{simple ghost number} of $kG$ is the smallest $n$ such 
that every composite of $n$ simple ghosts in $\stmod{kG}$ is trivial.
The point here is that $\stmod{kG}$ is the thick subcategory generated
by the simple modules, so this is exactly analogous to the ghost number,
with the trivial module $k$ replaced by the set of all simple modules.
Moreover, for a $p$-group, $k$ is the only simple
module, so the two notions coincide.
In turns out that there is a close relationship between the simple
ghost number of $kG$ and the ghost number of $kP$, where $P$ is a
Sylow $p$-subgroup of $G$, and by studying both invariants at once
we can make many more computations.
Again, these are described in the detailed summary below.

One of the most important techniques in our work is the use of
induction and restriction, which brings us to the third and final
invariant that we study in this paper.
A \dfn{strong ghost} is a map $f$ whose restriction to every subgroup is a ghost,
or equivalently, such that $\sHom(\Omega^i k\up_H^G, f) = 0$ for each subgroup
$H$ of $G$ and each $i \in \Z$.
The \dfn{strong ghost number} of $kG$ is the smallest $n$ such that every composite
of $n$ strong ghosts in $\stmod{kG}$ is trivial.
This follows the same pattern as above, since $\stmod{kG}$ is the thick
subcategory generated by the test objects $k\up_H^G$.
Unlike the other invariants, one can show that the strong ghost number of $kG$
equals the strong ghost number of $kP$, and so it suffices to study $p$-groups.
Below we summarize our computations of and bounds on strong ghost numbers.

\medskip

The overall organization of the paper is as follows. 
In Section~\ref{se:background}, we introduce general concepts that will
be of use in the rest of the paper and recall some background material
on modular representation theory.
Sections~\ref{se:normal-Sylow}, \ref{se:princ-block-gen-by-k} and~\ref{se:cyclic-Sylow}
study both the ghost number and the simple ghost number, and
are distinguished by the assumptions placed on the group:
In Section~\ref{se:normal-Sylow}, we assume that the Sylow $p$-subgroup of $G$ is normal.
In Section~\ref{se:princ-block-gen-by-k}, we assume that $\thick k$ coincides with
the principal block.
And in Section~\ref{se:cyclic-Sylow}, we assume that the Sylow $p$-subgroup is cyclic.
Finally, in Section~\ref{se:strong-ghosts}, we study the strong ghost number.

Note that there is some overlap in the assumptions made in
Sections~\ref{se:normal-Sylow}, \ref{se:princ-block-gen-by-k} and~\ref{se:cyclic-Sylow}.
For example, in Section~\ref{ss:direct product} we study groups whose Sylow $p$-subgroup
is a direct factor, and these groups satisfy the assumptions of
Sections~\ref{se:normal-Sylow} and~\ref{se:princ-block-gen-by-k}.
This includes the case of $p$-groups.
And in Section~\ref{ss:normal-P}, we study groups with a cyclic normal Sylow $p$-subgroup,
and these satisfy the assumptions of all three sections.
In general, the assumptions made are independent, except that Sunil Chebolu and Jan Min\'{a}\v{c}
have an unpublished proof that when the Sylow $p$-subgroup is cyclic, $\thick k$
coincides with the principal block.
(This may be one of those results that is ``known to the experts''.)


\medskip

We now summarize the main results of each section in more detail.
In Section~\ref{ss:length-general}, working in a general triangulated category,
we define the Freyd length and Freyd number with respect to
a set $\bP$ of test objects.  The Freyd number generalizes the ghost number,
simple ghost number and strong ghost number defined above.
We also recall the closely related concept of length with respect to a 
projective class, and we prove general results about both of these invariants.
In Section~\ref{ss:stmod}, we recall the basics of the stable module category,
and in Section~\ref{ss:gh_classes} we formally introduce ghosts and simple
ghosts, specializing the Freyd length and Freyd number to these two situations.

In Section~\ref{se:normal-Sylow} we assume that our group $G$ has a normal Sylow $p$-subgroup $P$.
Under this assumption, in Section~\ref{ss:simple-class-normal} we show that a
map in $\StMod{kG}$ is a simple ghost if and only if its restriction to $P$ is a ghost,
and show that the simple ghost number of $kG$ is equal to the ghost number of $kP$.
It follows that when $P$ is normal, the simple generating hypothesis holds if and only if
$P$ is $C_2$ or $C_3$.
(We don't have a characterization of when the simple generating hypothesis holds in 
general, but we do know that it does not depend only on the Sylow $p$-subgroup.
See Section~\ref{ss:SL(2,p)}.)
In Section~\ref{ss:A_4}, we apply this result to the group $A_4$ at the prime $2$,
deducing that the simple ghost number is $2$ and that the ghost number is between
$2$ and $4$.
We also give an example of a ghost for $A_4$ whose restriction to the Sylow $p$-subgroup
is not a ghost.

In Section~\ref{se:princ-block-gen-by-k}, we focus on groups whose principal
block is generated by $k$ in the sense that $\stmod{B_0} = \thick k$ 
(or, equivalently, $\StMod{B_0} = \loc k$).
We show that this holds when the Sylow $p$-subgroup $P$ is a direct factor, 
in Section~\ref{ss:direct product}, using a result that shows that there is
an equivalence between $\stmod{kP}$ and $\thickC G k$.
This last result corrects an error in~\cite{GH per};
see the comments after Theorem~\ref{th:products}.
In Section~\ref{ss:B_0}, we show that if $\stmod{B_0} = \thick k$, 
then the ghost number of $kG$ is finite.
We prove this by using a comparison to the simple ghost number, which is finite for any $G$.
We conjecture that the ghost number is finite for general $G$.
This is related to a question proposed in~\cite{BCR}. (See Remark~\ref{rm:finite-ghost-number}.)
Still assuming that the principal block is generated by $k$, 
we show that the ghost number of $kG$ is greater than or equal
to the ghost number of $kP$, by first showing that the composite of
inducing up from $P$ to $G$ followed by projection onto the principal block
is faithful.
In Section~\ref{ss:dihedral}, working at the prime $2$,
we show that for a dihedral group $D_{2ql}$ of order $2ql$, with $q$ a power of $2$ and $l$ odd,
the principal block is generated by $k$ and the ghost number of $kD_{2ql}$ is
equal to the ghost number of the Sylow $2$-subgroup $D_{2q}$, which was shown
to be $\lfloor \frac{q}{2}+1 \rfloor$ in~\cite{Gh num}.
By computing the simple ghost lengths of modules in non-principal blocks,
we are also able to show that the simple ghost number of $kD_{2ql}$ is
again $\lfloor \frac{q}{2}+1 \rfloor$.

Section~\ref{se:cyclic-Sylow} studies the case when the Sylow $p$-subgroup $P$ is cyclic.
In Section~\ref{ss:normal-P}, we assume that $P$ is cyclic and normal,
and show that every simple module in the principal block is a suspension of the trivial module.  
It follows that $\stmod{B_0} = \thick k$ and that a map in $\thick k$ is a ghost
if and only if it is a simple ghost. 
Thus the simple ghost number of $kG$, the ghost number of $kG$ and the
ghost number of $kP$ are all equal.  Since $P$ is a cyclic $p$-group,
its ghost number is known~\cite{Gh in rep}.
In particular, this allows us to compute the ghost numbers of the dihedral groups
at an odd prime.
Combined with the results above,
this completes the computation of the ghost numbers of the dihedral groups, at any prime.
The group $SL(2, p)$ has a cyclic Sylow $p$-subgroup $P$, but it is not normal.
By studying the normalizer $L$ of $P$ and applying the results of Section~\ref{ss:normal-P} to $L$,
we show in Section~\ref{ss:SL(2,p)} that the simple generating hypothesis holds for $SL(2,p)$
over a field $k$ of characteristic $p$.
Along the way, we find that there is an equivalence $\stmod{kG} \to \stmod{kL}$,
but that the simple generating hypothesis does \emph{not} hold for $kL$.

In Section~\ref{se:strong-ghosts} we study strong ghosts.
We begin in Section~\ref{ss:sgh class} by showing that the strong ghost number of a
group algebra $kG$ equals the strong ghost number of $kP$, 
where $P$ is a Sylow $p$-subgroup of $G$.
Then we compute the strong ghost numbers of cyclic $p$-groups in Section~\ref{ss:sgn-C_p^r}.
Finally, in Section~\ref{ss:sgn-D_4q}, we show that the strong ghost number of a
dihedral $2$-group $D_{4q}$ is between $2$ and $3$, with the upper bound being
the non-trivial result.

\section{Background}\label{se:background}

In this section, we provide background material that will be used throughout the paper.
In Section~\ref{ss:length-general}, we define invariants of a triangulated category $\sfT$
which depend on a set $\bP$ of test objects, and prove general results about these invariants.
In Section~\ref{ss:stmod}, we recall some background results about the stable module
category of a finite group.
In Section~\ref{ss:gh_classes}, we apply the general theory to two sets of test objects
in the stable module category of a group, giving rise to invariants called the ghost number and
the simple ghost number.

\subsection{The generating hypothesis and related invariants}\label{ss:length-general}

We begin this section by stating the generating hypothesis with respect to 
a set of objects in a triangulated category and defining invariants, the 
Freyd length and the length, which measure the degree to which the 
generating hypothesis fails.
Motivated by this, we recall the definition of a projective class.
Then, working in a general triangulated category, we study the relationship 
between the lengths (and Freyd lengths) of an object with respect to different projective classes.
We also compare lengths in different categories by using the pullback projective class.

Let $\sfT$ be a triangulated category, and let $\bP$ be a set of objects in $\sfT$.
The thick subcategory generated by $\bP$, denoted $\thick {\bP}$,
is the smallest full triangulated subcategory of $\sfT$ 
that is closed under retracts and contains $\bP$.
It is easy to see that $\bP$ detects zero objects in $\thick {\bP}$, i.e.,
if $M \in \thick {\bP}$ and $[\Sigma^i \bP, M] = 0$ for all $P \in \bP$ and $i \in \Z$,
then $M \iso 0$.
Here we write $[\blank, \blank]$ for the hom-sets in $\sfT$.

The \dfn{generating hypothesis} for the set of test objects $\bP$
is the statement that
$\bP$ detects trivial \emph{maps} in $\thick \bP$, i.e.,
if $f$ is a map in $\thick {\bP}$ and $[\Sigma^i P, f] = 0$ for all $P \in \bP$ and $i \in \Z$,
then $f$ is the zero map~\cite{GH general}.

When the generating hypothesis for $\bP$ fails, there is a natural
invariant which measures the degree to which it fails.
Let $\cI$ denote the class of maps $f$ such that $[\Sigma^i P, f] = 0$ for all $P \in \bP$ and $i \in \Z$,
and write $\cI_t$ for such maps in $\thick \bP$.
The \dfn{Freyd length} $\Flen_{\bP}(X)$ of an object $X$ in $\thick \bP$ with respect to $\bP$
is the smallest number $n$ such that every composite $X \to X_1 \to \cdots \to X_n$
of $n$ maps in $\cI_t$ is zero.  
The \dfn{Freyd number} of $\sfT$ with respect to $\bP$ is the least upper bound of 
the Freyd lengths of the objects in $\thick \bP$.
With this terminology, the generating hypothesis holds for $\bP$ if and only if 
the Freyd number of $\sfT$ with respect to $\bP$ is 1.

It turns out to be fruitful to consider a related invariant, where none of the
objects are required to lie in $\thick \bP$.
The \dfn{length} $\len_{\bP}(X)$ of an object $X$ in $\sfT$ with respect to $\bP$
is the smallest number $n$ such that every composite $X \to X_1 \to \cdots \to X_n$
of $n$ maps in $\cI$ is zero, if this exists (which is the case when $X \in \thick \bP$).
This is clearly at least as big as the Freyd length, but has better formal
properties which make it easier to work with.
These properties are best expressed in terms of the \emph{projective class} generated by $\bP$.
To motivate the definition, note that $\smd{\bP}$ detects the same maps in 
$\sfT$ as $\bP$ does, 
where $\smd{\bP}$ denotes the closure of $\bP$ under retracts, sums,
suspensions and desuspensions.
Moreover, it is easy to show~(\cite{Chr}) that $\cP := \smd{\bP}$ and $\cI$ determine 
each other in the sense of the following definition:

\begin{de}
Let $\sfT$ be a triangulated category. 
A \dfn{projective class} in $\sfT$ consists of a class $\cP$ of objects of $\sfT$
and a class $\cI$ of morphisms of $\sfT$ such that:
\begin{enumerate}[(i)]
\item $\cP$ consists of exactly the objects $P$ such that every composite
      $P \to X \to Y$ is zero for each $X \to Y$ in $\cI$,
\item $\cI$ consists of exactly the maps $X \to Y$ such that every composite
      $P \to X \to Y$ is zero for each $P$ in $\cP$,
\item for each $X$ in $\sfT$, there is a triangle $P \to X \to Y \to \Sigma P$ with 
      $P$ in $\cP$ and $X \to Y$ in $\cI$.
\end{enumerate}
\end{de}

Our main examples will be projective classes of the form $(\smd{\bP}, \cI)$,
which we call the \dfn{(stable) projective class generated by $\bP$}.

Given a projective class $(\cP, \cI)$, there is a sequence of 
\dfn{derived projective classes} $(\cP_n,\cI^n)$~\cite{Chr}.
The ideal $\cI^n$ consists of all $n$-fold composites of maps in $\cI$, 
and $X$ is in $\cP_n$ if and only if it is a retract of an object $M$ that
sits inside a triangle $P \to M \to Q \to \Sigma P$ with $P \in \cP_1=\cP$ and 
$Q \in \cP_{n-1}$. 
For $n=0$, we let $\cP_0$ consist of all zero objects and $\cI^0$ consist of all
maps in $\sfT$. 

Extending the definition above to any projective class, we define
the \dfn{length} $\len_{\cP}(X)$ of an object $X$ in $\sfT$ with respect to $(\cP, \cI)$
to be the smallest number $n$ such that every map in $\cI^n$
with domain $X$ is trivial.
The fact that each pair $(\cP_n, \cI^n)$ is a projective class implies that
the length of $X$ is equal to the smallest $n$ such that $X \in \cP_n$.
When $\cP = \smd{\bP}$, we write $\len_{\bP}(X)$ as above.

We note that different sets of objects can generate the same projective
class but different thick subcategories, so the Freyd length depends on
the choice of generating set $\bP$, 
not just on the projective class $\smd{\bP}$ it generates.

The following lemma is a direct consequence of the definition of a projective class.
This idea is used in comparing the ghost length and the simple ghost length of a module.

\begin{lem}\label{le:lengths}
Let $\sfT$ be a triangulated category, and
let $(\cP,\cI)$ and $(\cQ,\cJ)$ be projective classes in $\sfT$.
Then we have the following relationships:
\begin{itemize}
\item
If $M$ has finite length with respect to $(\cP,\cI)$, then
\[ \len_{\cP_n} (M) = \bigg\lceil \frac{\len_{\cP} (M)}{n} \bigg\rceil. \]

\item
If $\cQ \subseteq \cP$, then
\[ \len_{\cP}(M) \leq \len_{\cQ}(M). \]

\item
If $\cQ \subseteq \cP_n$, then
\[ \len_{\cP} (M) \leq n \, \len_{\cP_n} (M) \leq n \, \len_{\cQ}(M).\]
\end{itemize}
\end{lem}

\begin{proof}
To show $\len_{\cP_n} (M) = \bigg\lceil \frac{\len_{\cP} (M)}{n} \bigg\rceil$,
we actually need to prove two inequalities:
\begin{equation}\label{eq:ineq}
 \len_{\cP_n} (M) \leq \bigg\lceil \frac{\len_{\cP} (M)}{n} \bigg\rceil
  \quad\text{and}\quad
  \len_{\cP} (M) \leq n \, \len_{\cP_n} (M).
\end{equation}

For the second inequality, let $\len_{\cP_n} (M) = m$.  
Then $M \in (\cP_n)_m \subseteq \cP_{mn}$,
which means that $\len_{\cP} (M) \leq mn$.
Equivalently, we can prove the inequality 
using the inclusion $\cI^{mn} \subseteq (\cI^n)^m$, i.e.,
if every $m$-fold composite of $n$-fold composites of maps in $\cI$ out of $M$ is trivial, 
then every $mn$-fold composite of maps in $\cI$ out of $M$ is trivial.

Using the inclusions the other way, i.e.,
$(\cP_n)_m \supseteq \cP_{mn}$ and $\cI^{mn} \supseteq (\cI^n)^m$,
one can prove that $\len_{\cP_n} (M) \leq \big\lceil \len_{\cP} (M) / n \big\rceil$.

The other inequalities in the lemma follow with similar proofs.
\end{proof}

The analog of Lemma~\ref{le:lengths} for Freyd lengths is a bit more
subtle because of the need to take into account the appropriate
thick subcategories.  
For example, if $(\smd{\bP}, \cI)$ and $(\smd{\bQ}, \cJ)$ are projective
classes and $\bQ \subseteq \bP$, then clearly $\cI \subseteq \cJ$.
But the inclusion $\thick \bQ \subseteq \thick \bP$ goes in the other
direction, so in general there is no inclusion between $\cI_t$ and $\cJ_t$.

Nevertheless, if we include assumptions which control the thick subcategories,
then most of the results go through.
We simply work with $(\cI_t)^n$ instead of $\cI^n$.
However, one difference is that we only have an inclusion
$(\cI_t)^{mn} \subseteq ((\cI^n)_t)^m$,
rather than an equality, and
as a result, we lose the first inequality from equation~\eqref{eq:ineq}.
In the next lemma, we give a result which we will use later.

\begin{lem}\label{le:Flengths}
Let $\sfT$ be a triangulated category, and
let $(\smd\bP,\cI)$ and $(\smd\bQ,\cJ)$ be projective classes in $\sfT$
generated by sets $\bP$ and $\bQ$.
If $\bP \subseteq \thick \bQ$, $\bQ \subseteq \smd{\bP}_n$
and $M \in \thick \bP$, then
\[ \Flen_{\bP} (M) \leq n \, \Flen_{\bQ} (M) . \]
\end{lem}

\begin{proof}
Let $m = \Flen_{\bQ}(M)$.
We must show that any composite $M = M_0 \to M_1 \to \cdots \to M_{mn}$ of maps in $\cI$
with the $M_i$ in $\thick \bP$ is zero.
The inclusion $\bP \subseteq \thick \bQ$ tells us that $\thick \bP \subseteq \thick \bQ$,
so these maps are in $\thick \bQ$.
The inclusion $\bQ \subseteq \smd{\bP}_n$ tells us that $\cI^n \subseteq \cJ$.
Thus the above composite is an $m$-fold composite of maps in $\cJ \cap \thick \bQ$,
and so is zero by the definition of $m$.
\end{proof}

Consider a triangle
\[ M' \lra M \lra M'' \lra \Sigma M' \]
in $\sfT$.
We know that $\len(M) \leq \len(M') + \len(M'')$ by~\cite[Note~3.6]{Chr}.
We will prove the analog for Freyd lengths.

\begin{lem}\label{le:Flengths-in-triangle}
Let $\sfT$ be a triangulated category with a set of test objects $\bP$, and let
$\cI_t$ be the class of maps in $\thick \bP$ that are trivial on $\bP$.
Let $M' \xrightarrow{\alpha} M \xrightarrow{\beta} M'' \to \Sigma M'$ be a triangle in $\sfT$.
If $M'$ and $M''$ have finite Freyd lengths, then
\[\Flen_{\bP}(M) \leq \Flen_{\bP}(M') + \Flen_{\bP}(M''). \]
\end{lem}

\begin{proof}
Let $n = \Flen_{\bP}(M')$ and $l = \Flen_{\bP}(M'')$.
We want to show that any map $\phi : M \to N$ in $(\cI_t)^{n+l}$ is trivial.
Write $\phi$ as $\phi_2 \phi_1$, where $\phi_1$ is in $(\cI_t)^n$
and $\phi_2$ is in $(\cI_t)^l$.
Then, since $\Flen_{\bP}(M') = n$, the composite $\phi_1 \alpha$ is trivial and
$\phi_1$ factors through $M''$:
\[
\xymatrix{
 M' \ar[r]^{\alpha} & M \ar[r]^{\beta} \ar[d]^{\phi_1} & M'' \ar@{-->}[dl]^{\psi} \\
                     & W \ar[d]^{\phi_2} \\
                     & N .
}
\]
Now since $\Flen_{\bP}(M'') = l$, the composite $\phi_2 \psi$ is trivial 
and so $\phi$ is trivial as well.
\end{proof}

Now we explain how to compare lengths in different categories, using the pullback projective class.

\begin{de}\label{de:pullback}
Let $U:\sfT \to \sfS$ be a triangulated functor between triangulated categories,
together with a left adjoint $F:\sfS \to \sfT$ that is also triangulated,
and let $(\cP,\cI)$ be a projective class in $\sfS$.
We define
\[ \cI' := \{ M \to N \text{ in }\sfT \text{ such that }UM \to UN\text{ is in } \cI \}
         = U^{-1}(\cI). \]
Then $\cI'$ forms the ideal of a projective class in $\sfT$ with relative projectives
\[ \cP' = \{ \text{retracts of } FP \text{ for } P \text{ in } \cP \}.\]
The projective class $(\cP', \cI')$ in $\sfT$ is called the \dfn{pullback}
of $(\cP, \cI)$ along the right adjoint $U$~\cite{rel ha}.
It is the projective class in $\sfT$ generated by the class of objects $F(\cP)$.
\end{de}
Note that if $(\cP, \cI)$ is a stable projective class,
i.e., $\cP$ (or equivalently $\cI$) is closed under suspension and desuspension,
then $(\cP', \cI')$ is also stable and $\cP' = \smd {F(\cP)}$.
One readily sees that the following relationships hold,
since $F$ sends $\cP$ into $\cP'$ and
$U$ sends $\cI'$ into $\cI$.

\begin{lem}\label{le:length-pullback}
Suppose we are in the above situation and
that $M \in \sfS$ and $N \in \sfT$. Then
\[ \len_{\cP} (M) \geq \len_{\cP'} (FM), \]
and, if the functor $U$ is faithful,
\[ \len_{\cP'} (N) \leq \len_{\cP} (UN). \xqedhere{165.5pt}\]
\end{lem}

\subsection{The stable module category}\label{ss:stmod}
Let $G$ be a finite group, and
let $k$ be a field whose characteristic $p$ divides the order of $G$.
The stable module category $\StMod{kG}$
is a quotient category of the module category $\Mod{kG}$.
\newcommand{\nn}{\mspace{-1mu}}
For $kG$-modules $M$ and $N$,
the hom-set $\sHom(\nn M, N)$ in $\StMod{kG}$ is the quotient
$\Hom(\nn M, N) \nn / \PHom(\nn M, N)$,
where $\PHom(M, N)$ consists of the maps 
that factor through a projective module.
Then $\StMod{kG}$ is a triangulated category with
triangles coming from short exact sequences in $\Mod{kG}$.
Two modules $M$ and $N$ are said to be \dfn{stably isomorphic} if they are isomorphic in the 
stable module category, and this holds if and only if their projective-free summands are
isomorphic as $kG$-modules.
We use the symbol $\iso$ for isomorphism as $kG$-modules, unless otherwise stated.
The desuspension $\Omega M$ of a module $M$ is defined to be the kernel
in any short exact sequence
\[ 0 \lra \Omega M \lra Q \lra M \lra 0,\]
where $Q$ is a projective $kG$-module.
Note that $\Omega M$ is well-defined in the stable module category, and
we denote by $\widetilde{\Omega} M$ the projective-free summand of $\Omega M$.
We write $\stmod{kG}$ for the full subcategory of finitely generated
modules in $\StMod{kG}$.
(More precisely, we include all modules which are stably isomorphic to finitely generated
$kG$-modules.)
We refer to~\cite{Carlson} for more background on $\StMod{kG}$.

Now let $P$ be a Sylow $p$-subgroup of $G$.
We consider the adjunction
\[ \up^G:  \StMod{kP} \rightleftarrows \StMod{kG}: \res = \down_P, \]
with $\up^G$ as a left adjoint.
We quote the following important facts in modular representation theory for further use:

\begin{lem}[\cite{Benson}]\label{le:background}
Let $G$ be a finite group, let $k$ be a field whose characteristic $p$ divides the order of $G$, 
and let $P$ be a Sylow $p$-subgroup of $G$. Then the following hold:
\begin{enumerate}[(i)]
\item
The restriction functor $\down_P: \StMod{kG} \to \StMod{kP}$ is faithful.
\item
Each $kG$-module $M$ is a summand of the module $M \down_P \up^G$.
\item
A $kG$-module $Q$ is projective if and only if its restriction $Q \down_P$ is projective.
\xqedhere{74.5pt}
\end{enumerate}
\end{lem}

\begin{thm}[Mackey's Theorem~\cite{Benson}]\label{th:mackey}
Let $L$ and $H$ be subgroups of $G$, and let $V$ be a $kH$-module.
Then
\[ 
  (V \up_H^G) \down_L \ \iso 
             \bigoplus_{s\in L\backslash G/H} (sV) \down_{L\cap sHs^{-1}} \up^L. 
\]
Here $sV=s\otimes V$ is the corresponding $k(sHs^{-1})$-module for $s \in G$, and
the sum is taken over the double coset representatives. \qed 
\end{thm}

\subsection{Ghost lengths and simple ghost lengths in $\StMod{kG}$}\label{ss:gh_classes}

By the \dfn{generating hypothesis} in $\StMod{kG}$,
we mean the generating hypothesis with respect to the set $\{ k \}$
containing only the trivial module.
Since Tate cohomology is represented by $k$,
the associated ideal $\cG$ consists of \dfn{ghosts} in $\StMod{kG}$, i.e.,
maps which induce the zero map in Tate cohomology.
Thus the generating hypothesis is the statement that there are no non-trivial
maps in $\thick k$ which induce the zero map in Tate cohomology.
The projective class $(\cF, \cG)$ generated by $k$ has $\cF = \smd{k}$,
summands of direct sums of suspensions and desuspensions of $k$.
We call $(\cF, \cG)$ the \dfn{ghost projective class}.
When we need to indicate the dependence on the group, we write $(\cF^G,\cG_G)$.

For a module $M \in \thick k$, its \dfn{ghost length} $\gl(M)$ is defined to be
its Freyd length with respect to $\{ k \}$, and
the \dfn{ghost number} of $kG$ is the Freyd number of $\StMod{kG}$ with respect to $\{ k \}$.
With this terminology, the generating hypothesis is the statement
that the ghost number of $kG$ is $1$.

Since the restriction functor preserves the trivial module,
we can induce up a (non-trivial) ghost from a subgroup of $G$ to
get a (non-trivial) ghost $G$.
This provides a very convenient tool when we study $p$-groups.
However, the inducing up technique has limited use for a general finite group,
since the ghosts, when induced up, do not always land in $\thick k$,
which is often smaller than $\stmod{kG}$.

In general, the stable module category is generated by the set $\bS$ of simple modules.
This suggests that we examine the projective class $(\cS, \sG)$ generated by $\bS$,
which we call the \dfn{simple ghost projective class}, and
compare it to the ghost projective class.
Here $\cS = \smd {\bS}$, and the maps in $\sG$ are called \dfn{simple ghosts}.
The \dfn{simple generating hypothesis} for $kG$ is the generating hypothesis
with respect to $\bS$.
The Freyd length (respectively number) with respect to $\bS$ will be called
the \dfn{simple ghost length}  (respectively \dfn{number}).
For $M \in \stmod{kG}$, the simple ghost length is denoted by $\sgl(M) = \Flen_{\bS}(M)$.
Note that while the ghost length is only defined for $M \in \thick{k}$,
the simple ghost length is defined for all $M \in \stmod{kG}$ since
$\stmod{kG} = \thick \bS$.
If $G$ is a $p$-group, then $\bS = \{ k \}$, so
the simple ghost projective class and the ghost projective class coincide.

\begin{rmk}\label{rm:finite-simple-number}
The radical series of a $kG$-module $M$ gives a construction of $M$ using simple modules,
showing that $\len_{\bS}(M)$ is at most the radical length of $M$,
since the pair $(\cS_n, \sG^n)$ is a projective class,
as described in Section~\ref{ss:length-general}.
Therefore,
\[
\sgl(M) = \Flen_{\bS}(M) \leq \len_{\bS}(M) \leq \rl(M) \leq \rl(kG).
\]
This shows that the simple ghost number of $kG$ is finite.
In particular, for $P$ a $p$-group, the ghost number of $kP$ is finite.
In Conjecture~\ref{cj:finite-ghost-number} we assert that the ghost
number of $kG$ is always finite, but this is an open question.
\end{rmk}

In the last section of the paper, we will study another projective class in $\StMod{kG}$,
which is called the strong ghost projective class.

\section{Groups with normal Sylow $p$-subgroups}\label{se:normal-Sylow}

In this section, we assume that our group $G$ has a normal Sylow $p$-subgroup $P$.
Under this assumption, in Section~\ref{ss:simple-class-normal} we show that 
the simple ghost number of $kG$ is equal to the ghost number of $kP$.
In Section~\ref{ss:A_4}, we apply this result to the group $A_4$ at the prime $2$,
deducing that the simple ghost number is $2$ and that the ghost number is between
$2$ and $4$.

\subsection{The simple projective class as a pullback}\label{ss:simple-class-normal}

In this section, we show that the simple ghost projective class in $\StMod{kG}$ is the pullback of 
the ghost projective class in $\StMod{kP}$,
under the assumption that the Sylow $p$-subgroup $P$ is normal in $G$.
Then we show that simple ghost lengths in $\StMod{kG}$ are the same as ghost lengths in $\StMod{kP}$.
The main result of this section should be viewed as the stabilised version of the next lemma:

\begin{lem}[{\cite[Lemma 5.8]{A}}]\label{le:res-simple}
Let $k$ be a field of characteristic $p$, and
let $G$ be a finite group with a normal Sylow $p$-subgroup $P$.
Let $M$ be a $kG$-module.
Then $\rad(M) \down_P = \rad(M \down_P)$.
It follows that the radical sequence of $M$ coincides with that of $M \down_P$.
In particular, $M$ is semisimple if and only if $M \down_P$ is.\qed
\end{lem}

We write $(\smd{\cF^P \up^G},\, \res^{-1}(\cG_P))$ for the pullback of $(\cF^P,\cG_P)$ along the restriction functor.
Then, by Lemma~\ref{le:background}(ii), we have $\cF^G \subseteq \smd{\cF^P\up^G}$.
Equivalently, $\res^{-1}(\cG_P) \subseteq \cG_G$, i.e.,
if a map in $\StMod{kG}$ restricts to a ghost in $\StMod{kP}$,
then it is a ghost.
(Note that we write $\down_P$ for the restriction functor except when
considering preimages, in which case we write $\res^{-1}$.)
We can describe $\res^{-1}(\cG_P)$ more precisely when $P$ is normal in $G$.

\begin{thm}\label{th:normal}
Let $k$ be a field of characteristic $p$, and
let $G$ be a finite group with a normal Sylow $p$-subgroup $P$.
Then the projective classes $(\cS, \sG)$ and $(\smd{\cF^P\up^G}, \res^{-1}(\cG_P))$
in $\StMod{kG}$ coincide, and
for $M \in \stmod{kG}$ and $L \in \stmod{kP}$, we have
\[ \sgl(M) = \gl(M \down_P) \quad\text{and}\quad \gl(L) = \sgl(L \up^G) .\]
Hence
\[\text{simple ghost number of } kG = \text{ghost number of } kP.\]
In particular, the simple generating hypothesis holds for $kG$ if and only if
$P \iso C_2$ or $P \iso C_3$.
\end{thm}

The first claim of the theorem is saying that a map in $\StMod{kG}$ is a simple
ghost if and only if its restriction to $P$ is a ghost.

\begin{proof}
We first show that both functors $\up^G$ and $\res = \down_P$ preserve the test objects.
The containment $\res(\cS) \subseteq \cF^P$ follows directly from Lemma~\ref{le:res-simple}.
To see that $\smd{\cF^P \up^G} \subseteq \cS$, 
by Lemma~\ref{le:res-simple} it suffices to check that
$k \up^G \down_P \iso \oplus k$,
and this is true by Mackey's theorem (Theorem~\ref{th:mackey}).
Finally, by Lemma~\ref{le:background}(ii),
we have inclusions $\cS \subseteq \smd{\res(\cS)\up^G} \subseteq \smd{\cF^P\up^G}$,
hence $\cS = \smd{\cF^P\up^G}$.
It follows immediately that $\sG = \res^{-1}(\cG_P)$, and so $\sG \down_{P} \subseteq \cG_P$.
Note that we also have that $\cG_P \up^{G} \subseteq \sG$, using that $\res(\cS) \subseteq \cF^P$
and that $\up^G$ is right adjoint to restriction.

We now prove that $\sgl(L \up^G) = \gl(L)$, with the other equality following similarly.
Since the induction functor takes a non-trivial ghost in $\stmod{kP}$ into
a non-trivial simple ghost in $\stmod{kG}$,
we get $\sgl(L \up^G) \geq \gl(L)$ for $L \in \stmod{kP}$.

To show that $\sgl(L \up^G) \leq \gl(L)$,
we claim that the natural isomorphism $\alpha: \sHom_G(L \up^G, M) \to \sHom_P(L, M \down_P) $
takes simple ghosts to ghosts.
Indeed, if $g: L \up^G \to M$ is a simple ghost,
then the morphism $\alpha(g)$ is the composite
$L \xrightarrow{\eta} L \up^G \down_P \xrightarrow{g \down_P} M \down_P$, and is a ghost. 
It follows that $\sgl(L \up^G) \leq \gl(L)$.
\end{proof}

\begin{rmk}
One can also consider the \emph{unstable} projective classes generated by the simple modules
in $\StMod{kG}$ and $\StMod{kP}$.
We write $(\cS_u, \sG_u)$ for the unstable projective class generated by the simple modules in $\StMod{kG}$ and
$(\cF_u, \cG_u)$ for the unstable projective class generated by the trivial module in $\StMod{kP}$.
Here $\cS_u$ consists of retracts of direct sums of simple modules in $\StMod{kG}$ and
$\cF_u$ consists of direct sums of the trivial module in $\StMod{kP}$.

By the proof of Theorem~\ref{th:normal}, the projective classes $(\cS_u, \sG_u)$ and $(\smd{\cF_u \up^G}, \res^{-1}(\cG_{u}))$
are the same in $\StMod{kG}$.
And since $P$ is a $p$-group,
the radical length of a projective-free $kP$-module $L$ is equal to
the length of $L$ with respect to the projective class $(\cF_u, \cG_u)$~\cite[Proposition~4.5]{Gh num}.
It follows that, for $M \in \stmod{kG}$ and $L \in \stmod{kP}$,
if $M \down_P$ and $L$ are projective-free modules, then 
\[ \len_{\cS_u}(M) = \rl(M \down_P) \quad\text{and}\quad \rl(L) = \len_{\cS_u}(L \up^G) .\]
Hence we see that Theorem~\ref{th:normal} and Lemma~\ref{le:res-simple} are
stable and unstable versions of each other.
\end{rmk}

\begin{rmk}
When the Sylow $p$-subgroup is not normal, there is no obvious
relationship between the simple ghost number of $kG$ and the
ghost number of its Sylow $p$-subgroup, or between their radical lengths.
See Section~\ref{ss:SL(2,p)} for more discussion.
\end{rmk}

\subsection{The group $A_4$ at the prime $2$}\label{ss:A_4}
In this section, we show that 
in general the restriction functor from a finite group $G$ 
to a Sylow $p$-subgroup $P$ does not preserve ghosts. 
We also compute the simple ghost number of $kA_4$ at the
prime $2$ and give bounds on its ghost number.

Let $G$ be $A_4$, the alternating group on $4$ letters, and set $p=2$, 
so $P=V$, the Klein four group, is normal in $A_4$. 
It is known that $\thickC {A_4\!} k = \stmod{kA_4}$~\cite{GH per}. 
For convenience, we assume that $k$ contains a third root of unity $\zeta$,
i.e., $\F_4 \subseteq k$. 
Then $k \up_V^G\iso k \oplus k_{\zeta} \oplus k_{\bar{\zeta}}$. 
Here $k_{\zeta}$ is the one-dimensional module with 
the cyclic permutation $(123)$ acting as $\zeta$ and 
elements of even order acting as the identity, 
and similarly for $k_{\bar{\zeta}}$. 
Note that by Lemmas~\ref{le:background}(ii) and~\ref{le:res-simple},
these are all the simple $kA_4$-modules, i.e., $\bS = \{ k, k_{\zeta}, k_{\bar{\zeta}} \}$.
By Theorem~\ref{th:normal}, a map restricts to a ghost in $\stmod{kV}$ if and only if
it is a simple ghost in $\stmod{kA_4}$.
Since $k_{\zeta} \ncong \widetilde{\Omega}^ik$ for all $i \in \Z$, 
the class of $kA_4$-modules $\cF = \smd{k}$ is strictly contained in $\cS = \smd{\bS}$, or equivalently,
simple ghosts are strictly contained in ghosts. 
Therefore, there exists a ghost in $\stmod{kA_4}$ which does not restrict
to a ghost in $\stmod{kP}$.

For a specific example, we consider the connecting map
$\gamma : k_{\zeta} \to \Omega k_{\zeta}$ 
in the Auslander-Reiten triangle~\cite[Section~4.12]{Benson}
\[
  \Omega^2 k_{\zeta} \lra  E  \lra k_{\zeta} \llra{\gamma} \Omega k_{\zeta}
\]
associated to the simple module $k_{\zeta}$.
Since $\gamma$ is stably non-trivial, it is not a simple ghost.
But since $k_{\zeta} \not\in \cF$, the map is a ghost, by~\cite[Theorem~2.1]{GH split}.

We now compute the simple ghost number of $kA_4$ and give bounds
on the ghost number.
We are able to get an upper bound for the ghost number of $kA_4$,
since the simple modules have bounded ghost lengths.

\begin{pro}\label{pr:A_4}
Let $k$ be a field of characteristic $2$.
Assume that $k$ contains a third root of unity $\zeta$. Then
\[
\text{simple ghost number of }kA_4 = \text{ghost number of } kV = 2,
\]
and
\[
2 \leq \text{ghost number of }kA_4 \leq 4.
\]
\end{pro}

\begin{proof}
By Theorem~\ref{th:normal}, the simple ghost number of $kA_4$ is equal
to the ghost number of $kV$, which is known to be $2$ (see~\cite{Gh in rep}).

Since $\stmod{kA_4} = \thick k$ and every simple ghost is a ghost,
the ghost number of $kA_4$ is at least $2$.
On the other hand,
there is a short exact sequence
\[ 0 \lra \widetilde{\Omega}^2 k \lra \widetilde{\Omega} k_{\zeta} \oplus \widetilde{\Omega} k_{\bar{\zeta}} \lra k \lra 0 \]
in $\modu{kA_4}$ (see~\cite[Section~4.17]{Benson}).
It follows that $\bS \subseteq \cF_2$.
Thus, by Lemma~\ref{le:Flengths},
the ghost number of $kA_4$ is at most twice the simple ghost number.
\end{proof}

Note that $\stmod{kA_4} = \thick k$.
In the next section, we prove finiteness under a weaker hypothesis.

\section{Groups whose principal block is generated by $k$}\label{se:princ-block-gen-by-k}

In this section, we further our study of the ghost number of a group
algebra $kG$ by making use of the fact that the thick subcategory
$\thick k$ generated by $k$ is contained in $\stmod{B_0}$, where
$B_0$ is the principal block of $kG$, and
$\stmod{B_0}$ consists of modules in $\stmod{kG}$ whose projective-free summands
are in the principal block $B_0$.
The reader is referred to~\cite{A} and~\cite{Benson} for background on block theory.

We focus on the case in which $\thick k = \stmod{B_0}$.
In Section~\ref{ss:direct product}, we show that this holds
when the Sylow $p$-subgroup $A$ is a direct factor.
In this situation, we prove that $\stmod{kA}$ is equivalent 
to $\thickC{G}{k}$, and use these results to show that
the ghost numbers of $kA$ and $kG$ agree.

In Section~\ref{ss:B_0},
we show that when $\thick k = \stmod{B_0}$, the ghost number of $kG$ is finite.
The finiteness of the ghost number remains an open question without this hypothesis.
We also show that in general the composite of functors
\[ e_0(\blank \up^G): \StMod{kP} \lra \StMod{B_0} \]
is faithful, where $P$ is a Sylow $p$-subgroup of $G$
and $e_0$ is the principal idempotent,
which allows us to prove that the ghost number of $kG$ is at least 
as large as the ghost number of $kP$ when $\thick k = \stmod{B_0}$.
We quote Theorem~\ref{th:principal} which provides
conditions equivalent to 
$\thick k = \stmod{B_0}$.

Finally, in Section~\ref{ss:dihedral}, we use this material to compute
the ghost numbers of the dihedral groups at the prime $2$.
In addition, we give a block decomposition of each dihedral group
and compute its simple ghost number.

\subsection{Direct products}\label{ss:direct product}

In this section, we study the ghost number of certain direct products,
making a slight correction to a result in~\cite{GH per}.
Let $k$ be a field of characteristic $p$, 
and $G=A\times B$ with $A$ being a $p$-group, 
and the order of $B$ being coprime to $p$. 
(That is, $A$ is the Sylow $p$-subgroup of $G$.)
Write $i:A \to A\times B$ for the inclusion of $A$ into $G$ and 
$\pi:A\times B \to A$ for the projection onto $A$. Then $\pi i=\id_A$. 

We will prove the following result.
Recall that for a class $\bP$ of objects, 
$\loc \bP$ denotes the localizing category generated by $\bP$, 
i.e., the smallest full triangulated subcategory that is closed under 
arbitrary coproducts and retracts and contains $\bP$. 

\begin{thm}\label{th:products}
Let $k$ be a field of characteristic $p$, and let $G=A\times B$ with $A$ a $p$-group
and the order of $B$ coprime to $p$. 
Then the projection $\pi \colon G \to A$ induces a triangulated functor 
$\pi^* \colon \StMod{kA} \to \StMod{kG}$ that preserves the trivial representation $k$, 
and it restricts to triangulated equivalences
\[
  \pi^*:\StMod{kA} \lra \locC G k 
\]
and
\[
  \pi^*:\stmod{kA} \lra \thickC G k .
\]
The inverse functors are the restriction functors.
Moreover, the image of $\pi^*$ consists of the $kG$-modules
whose projective-free summands have trivial $B$-actions.
\end{thm}

This theorem corrects the statement of Lemma~4.2 in~\cite{GH per}, 
which has $\StMod{kG}$ in place of $\locC G k$ and $\stmod{kG}$ in
place of $\thickC G k$.
That statement is false whenever $B$ is non-trivial.  
The problem is that the restriction functor $\StMod{kG} \to \StMod{kA}$
is not full.
For example, writing $kB$ for the $kG$-module on which $A$ acts trivially,
note that $kB \iso k \up_A^G$.
Then one can see that the dimension of $\sHom_G(kB, kB)$ is $|B|$,
while the dimension of $\sHom_A(kB\down_A, kB\down_A)$ is $|B|^2$.
The correction is simply to restrict attention to $\locC G k$.
The uses of Lemma~4.2 in~\cite{GH per} can be replaced with the above theorem
and the fact that $\thickC{G}{k} = \stmod{B_0}$ (Corollary~\ref{co:thick=B0} below), so
all of the main results of~\cite{GH per} are correct.

\begin{proof}[Proof of Theorem]
We first note that
the functor $\pi^*:\Mod{kA} \to \Mod{kG}$ induced by $\pi:G \to A$ passes down to the stable module categories.
To prove this, it suffices to show that if $P$ is a projective $kA$-module, then
$\pi^* P$ is projective.  
Since $\pi i = \id$, the restriction of $\pi^* P$ to $A$ is $P$,
and since $A$ is the Sylow $p$-subgroup of $G$, it follows from Lemma~\ref{le:background}(iii)
that $\pi^* P$ is projective.
It is easy to see that the functor $\pi^* \colon \StMod{kA} \to \StMod{kG}$ is triangulated and
preserves coproducts and the trivial representation. 

Let $\im(\pi^*)$ be the essential image of $\pi^*$ in $\StMod{kG}$. 
The modules in $\im(\pi^*)$ are exactly those whose projective-free summands
have trivial $B$-actions. 
It follows that $\pi^*$ is full and that $\im(\pi^*)$ is closed under coproducts. 
Since $i^* \pi^*=\id$, the functor $\pi^*$ is also faithful.
Thus $\pi^*$ induces a triangulated equivalence between $\StMod{kA}$ and $\im(\pi^*)$. 
Because $\StMod{kA}=\locC A k$ and $\pi^*$ is triangulated, we get that 
$\im(\pi^*)$ is contained in $\locC G k$ and that $\im(\pi^*)$ is triangulated. 
Hence $\im(\pi^*)=\locC G k$, and we get the triangulated equivalence 
$\pi^*:\StMod{kA} \to \locC G k$. 
Clearly, the restriction functor $i^*: \locC G k \to \StMod{kA}$ on the localizing 
subcategory generated by $k$ is inverse to $\pi^*$. 
Restricting to compact objects, we get the equivalence
$\pi^*:\stmod{kA} \to \thickC G k$,
since $\thick k$ consists of exactly the compact objects in $\loc k$ by~\cite[Lemma~2.2]{Neeman}.
\end{proof}

As a corollary, we can compute the ghost number of $kG$.

\begin{cor}\label{co:ghost no}
In the same set-up as above, the following holds:
\[ \text{ghost number of }kG = \text{ghost number of }kA. \]
In particular, the generating hypothesis holds for $kG$ if and only if 
it holds for $kA$ if and only if $A$ is $C_2$ or $C_3$.\qed
\end{cor}

To show that $\thick k = \stmod{B_0}$,
we compute the principal block idempotent using the next formula,
which makes use of the following terminology:
we say that a group element $g$ is \dfn{$p$-regular} if its order is not divisible by $p$;
otherwise it is said to be \dfn{$p$-singular};
an exception is that the identity element $1$ is both $p$-regular and $p$-singular.
We write $G_p$ for the set of $p$-singular elements in $G$ and
$G_{p'}$ for the set of $p$-regular elements.

\begin{thm}[{\cite[Theorem 1]{B}}]\label{th:B0 idem}
Let $k$ be a field of characteristic $p$,
let $G$ be a finite group, and
let $e_0 = \sum \epsilon_g g$ be the principal block idempotent in $kG$
with each $\epsilon_g$ in $k$.
Then
\[ \epsilon_g = |\{ (u,s) \in G_p \times G_{p'} \st us=g \}| \,\, |G_{p'}|^{-1} \]
for a $p$-regular element $g \in G$, and
$\epsilon_g = 0$ if $g$ is not $p$-regular.\qed
\end{thm}

\begin{cor}\label{co:thick=B0}
Let $k$ be a field of characteristic $p$.
Let $G = A \times B$, with $A$ being the Sylow $p$-subgroup of $G$.
Then $\thickC{G}{k} = \stmod{B_0}$ and $\locC{G}{k} = \StMod{B_0}$.
\end{cor}

By Theorem~\ref{th:principal}, the conditions
$\thickC{G}{k} = \stmod{B_0}$ and $\locC{G}{k} = \StMod{B_0}$ are equivalent.

\begin{proof}
We compute that the principal idempotent $e_0$ is $\frac{1}{|B|}(\sum_{b\in B}b)$,
using Theorem~\ref{th:B0 idem}.
Since $b e_0 = e_0$ for each $b \in B$,
the projective-free modules in $\stmod{B_0}$ and $\StMod{B_0}$ all have trivial $B$ actions.
Thus, by Theorem~\ref{th:products}, the claim follows.
\end{proof}

One can also prove the corollary using Theorem~\ref{th:principal}.

Note that the only simple module in $\stmod{B_0}$ is the trivial module $k$ in this case.
Indeed, since $A \leq G$ is normal, a simple module $S$ has trivial $A$-action (Lemma~\ref{le:res-simple});
and if $S$ is in $\stmod{B_0}$, then it has trivial $B$-action too, 
by Theorem~\ref{th:products}.
Hence $S$ is the trivial module $k$.

\begin{rmk}
One can check that the algebra map 
$kA \to k(A \times B) \xrightarrow{e_0} e_0(k(A\times B))$ is an isomorphism.
It induces the equivalence $\stmod{B_0} \to \stmod{kA}$ with inverse $\pi^*$.
This also explains why we need to shrink the domain of the functor $i^*$
to get an equivalence.
\end{rmk}

We combine the discussion in Section~\ref{ss:simple-class-normal} and the results of this section
in the next proposition.
\begin{pro}\label{pr:direct product}
Let $k$ be a field of characteristic $p$.
Let $G = A \times B$, with $A$ being the Sylow $p$-subgroup of $G$.
Then, for $M \in \stmod{B_0}$,
\[  \gl(M) = \sgl(M) = \gl ( M\down_A ), \]
and for $N \in \stmod{kA}$,
\[ \sgl( N\up ) = \gl(N) = \gl( e_0(N\up) ) = \sgl ( e_0(N\up) ) . \]
\end{pro}

\begin{proof}
Since the trivial module $k$ is the only simple module in $\stmod{B_0}$,
$\gl(M) = \sgl(M)$ for $M \in \stmod{B_0}$.
The equalities $\sgl(M) = \gl( M\down_A )$ and $\sgl( N \up ) = \gl(N)$ are from Theorem~\ref{th:normal}.
That $\gl(N) = \gl(e_0(N\up))$ for $N \in \stmod{kA}$ is a result of Theorem~\ref{th:products},
as one checks that the functor $e_0 (\blank \up)$ is isomorphic to the equivalence
$\pi^*: \stmod{kA} \to \thickC{G}{k}$.
The last equality is a special case of the first.
\end{proof}

Note that one can't expect $\gl(N) = \gl( e_0(N\up) )$ for groups
that aren't direct products, even when $\thick k = \stmod{B_0}$.
For example, this fails for $A_4$, using the discussion in Section~\ref{ss:A_4}
and the fact that $e_0 = 1$ in this case.

\subsection{Finiteness of the ghost number and a lower bound}\label{ss:B_0}
Let $G$ be a finite group,
let $k$ be a field whose characteristic $p$ divides the order of $G$, and
let $P$ be a Sylow $p$-subgroup of $G$.
In this section, assuming that $\thick k = \stmod{B_0}$,
we prove that the ghost number of $kG$ is finite (Theorem~\ref{th:finite-gh-num})
and is greater than or equal to the ghost number of $kP$ 
(Proposition~\ref{pr:faithful}).

\begin{thm}\label{th:finite-gh-num}
Let $k$ be a field of characteristic $p$, and
let $G$ be a finite group with Sylow $p$-subgroup $P$.
Suppose that $\thickC G k = \stmod{B_0}$.
Then the ghost number of $kG$ is finite.
\end{thm}

In particular, the theorem holds for any $p$-group $G$,
since then $\thickC G k = \stmod{B_0} = \stmod{kG}$,
recovering~\cite[Theorem~4.7]{Gh in rep}.
Our proof follows the approach used in Proposition~\ref{pr:A_4}
for the alternating group $A_4$.

\begin{proof}
Recall that the simple ghost number of $kG$ is finite (Remark~\ref{rm:finite-simple-number}).
Since there are no non-zero maps between different blocks,
it follows that the Freyd number of $kG$ with respect to $\bQ$ is finite,
where $\bQ$ is the set of simple modules in the principle block.
On the other hand, since $\thickC G k = \stmod{B_0}$ and $\bQ$ is a finite set,
we have $\bQ \subseteq \cF_n = \smd{\bP}_n$ for some $n$, where $\bP = \{ k \}$.
It then follows from Lemma~\ref{le:Flengths} that the ghost number of $kG$ is 
bounded above by $n$ times the Freyd number of $kG$ with respect to $\bQ$,
and thus is finite.
\end{proof}

We call the Freyd number of $kG$ with respect to $\bQ$ 
the \dfn{simple ghost number of $B_0$}.

\begin{rmk}\label{rm:finite-ghost-number}
Note that each $M \in \thick k$ has finite ghost length.
But we need to find a universal upper bound to prove finiteness of the
ghost number.
One idea is to look at the radical sequence as was done for $p$-groups
in~\cite{Gh in rep}.
When $\thickC G k = \stmod{B_0}$,
the simple modules that can appear in the radical sequence for $M \in \thickC G k$
all have finite ghost lengths.
However, whether the ghost number is finite when $\thickC G k \neq \stmod{B_0}$ remains open,
since we cannot answer the following question proposed in~\cite{BCR}:
does there exist a simple module in the principal block with vanishing Tate cohomology?
Indeed, if there exists a simple module in $\stmod{B_0}$ but not in $\thickC G k$,
and its Tate cohomology does not vanish,
then it can appear in the radical sequence of a module $M \in \thickC G k$.
Hence the proof here does not apply to the case where
$\thickC G k \neq \stmod{B_0}$.
\end{rmk}

We state the question in the general case as a conjecture:
\begin{conj}\label{cj:finite-ghost-number}
Let $G$ be a finite group, and let $k$ be a field whose characteristic divides the order of $G$.
Then the ghost number of $kG$ is finite.
\end{conj}

Now we determine a lower bound for the ghost number of $kG$.
Note that for a group $G$ with subgroup $H$, the induction functor sends
ghosts to ghosts and is faithful.
However, induction does not preserve $\thick k$ in general,
so this technique is of limited use in computing the ghost number of $kG$.
To try to remedy this, we can consider the composite $e_0(\blank\up^G)$
of induction with projection onto the principal block.
This will provide us with a ghost in $\thickC G k$ if we assume that
$\thickC G k = \stmod{B_0}$.

Note that we have adjunctions
\[ \up^G:   \stmod{kH} \rightleftarrows \stmod{kG}: \down_H \quad\text{and}\quad
   e_0(\blank): \stmod{kG} \rightleftarrows \stmod{B_0}: j , \]
where $j$ denotes the inclusion.
We show that the composite $e_0(\blank\up^G)$ is faithful in the case where
$H$ is a Sylow $p$-subgroup.

\begin{pro}\label{pr:faithful}
Let $k$ be a field of characteristic $p$, and
let $G$ be a finite group with Sylow $p$-subgroup $P$.
Then the functor
\[ e_0(\blank \up^G): \stmod{kP} \lra \stmod{B_0} \]
is faithful.
In particular, if $\thickC G k = \stmod{B_0}$, then
\[\text{ghost number of } kG \geq \text{ghost number of } kP.\]
\end{pro}

We don't know of a counterexample to the last inequality.

\begin{proof}
It suffices to show that the unit map 
\begin{eqnarray*}
   M & \longrightarrow & j(e_0(M\up^G)) \down_P  \\
   m & \longmapsto & e_0 \otimes m
\end{eqnarray*}
of the composite adjunction is split monic.

It is well known that $\up^G$ is both left and right adjoint to $\down_P$, with
unit map $\eta: M \to M\up^G\down_P$ sending $m$ to $1 \otimes m$, and
counit map $\epsilon: M\up^G\down_P \to M$ sending $g \otimes m$ to $gm$ if $g \in P$
and to $0$ if $g \not \in P$.

The unit map for the adjunction
$e_0(\blank) : \stmod{kG} \rightleftarrows \stmod{B_0}: j$
is the natural projection $N \to j(e_0 N)$ by left multiplication by $e_0$.
Since the stable module category $\stmod{kG}$ decomposes into blocks,
it is easy to check that $e_0(\blank)$ is also right adjoint to $j$,
with counit the natural inclusion $j(e_0 N) \to N$.

The composite
\[ M \lra j(e_0(M\up^G))\down_P \lra M\up^G\down_P \llra{\epsilon} M \]
sends $m$ to $\epsilon(e_0 \otimes m)$.
We show that it is an isomorphism.
Since $P$ is a $p$-subgroup of $G$ and
the only possible non-zero coefficient $\epsilon_h$ for $h \in P$ is
$\epsilon_1 = |G_{p'}|^{-1}$ by Theorem~\ref{th:B0 idem},
one sees that $\epsilon(e_0 \otimes m) = \epsilon_1 m$.
But $\epsilon_1$ is invertible in $k$, so
the composite is an isomorphism.
It follows that $M \to (e_0(M\up^G))\down_P$ is split monic and
the functor $e_0(\blank\up^G)$ is faithful.

It is clear that the composite $e_0(\blank \up^G)$ preserves ghosts.
Hence $\gl(e_0 (L \up^G)) \geq \gl(L)$ for $L \in \stmod{kP}$, 
and the ghost number of $kG$ is greater than or equal to the ghost number of $kP$.
\end{proof}

We quote the next theorem to end this section.
It provides conditions for checking whether $\thickC G k = \stmod{B_0}$.
Recall that a finite group is said to be \dfn{$p$-nilpotent} if
$G_{p'}$, the set of $p$-regular elements of $G$, forms a subgroup.

\begin{thm}[{\cite[Theorem~1.4]{Principal}}]\label{th:principal}
Let $G$ be a finite group, and let $k$ be a field of characteristic $p$.
Then $\thick k = \stmod{B_0}$ 
if and only if
$\loc k = \StMod{B_0}$
if and only if
the centralizer of every element of order $p$ is $p$-nilpotent.\qed
\end{thm}

\subsection{Dihedral groups at the prime 2}\label{ss:dihedral}
Let $G=D_{2ql}$ be a dihedral group of order $2ql$,
where $q$ is a power of $2$ and $l$ is odd, with presentation
$D_{2ql}= \langle x,y \st x^{ql}=y^2=(xy)^2=1 \rangle$.
Let $k$ be a field of characteristic $2$.
In this section, 
we will determine the ghost number and simple ghost number of $kD_{2ql}$
by analyzing the blocks.
(See Theorem~\ref{th:dihedral-odd} for the ghost number of $kD_{2ql}$ at an odd prime.)

We can compute the principal block idempotent of $kD_{2ql}$ using Theorem~\ref{th:B0 idem}
and the fact that $l = 1$ in $k$.

\begin{lem}
The $2$-regular elements of $D_{2ql}$ are exactly those in the subgroup $C_l=\langle x^q \rangle$.
The principal idempotent is $e_0 = 1 + x^q + x^{2q} + \cdots + x^{(l-1)q}$.\qed
\end{lem}

We regard $D_{2q}$ as the subgroup of $D_{2ql}$ generated by $x^l$ and $y$,
and so we have a natural unital algebra map $\alpha: kD_{2q} \to kD_{2ql} \to e_0 kD_{2ql}$.
Note that $D_{2q}$ is a Sylow $2$-subgroup of $D_{2ql}$.

\begin{lem}\label{le:D_2ql-even}
The algebra map $\alpha: kD_{2q} \to e_0 kD_{2ql}$ is an isomorphism.
\end{lem}

\begin{proof}
As an algebra, $e_0 kD_{2ql}$ is generated by $e_0 x$ and $e_0 y$.
Clearly, $e_0 y = \alpha (y)$ is in the image of $\alpha$.
And since $l$ is odd and $e_0 x^q = e_0$, we see that
$e_0 x = e_0 x^{kl}$ for some integer $k$.
Hence the map $\alpha$ is surjective.
Since $e_0 kD_{2ql}$ is projective as a $kD_{2ql}$-module,
its dimension is at least $2q$, which equals the dimension of $kD_{2q}$,
so $\alpha$ has to be an isomorphism.
\end{proof}

As a corollary, we can compute the ghost number of $kD_{2ql}$.

\begin{cor}\label{co:D_2ql-even}
The thick subcategory generated by $k$ is the same as the principal block,
\[ \thickC {D_{2ql}} k = \stmod{e_0 kD_{2ql}} , \]
and the ghost number of $kD_{2ql}$ is $\lfloor \frac{q}{2}+1 \rfloor$.
\end{cor}

Note that in this case, the lower bound given by
Proposition~\ref{pr:faithful} is an equality.

\begin{proof}
Since $\alpha$ is an isomorphism, it induces an equivalence
$\stmod{kD_{2q}} \to \stmod{e_0 kD_{2ql}}$
sending $M$ to $e_0(M\up^{D_{2ql}}_{D_{2q}})$.
The first statement follows from the facts that this equivalence
sends $k$ to $k$ and that $\thickC {D_{2q}} k = \stmod{kD_{2q}}$.
It also follows that
\[\text{the ghost number of } kD_{2ql} = \text{the ghost number of } kD_{2q}.\]
The second statement then follows from~\cite[Corollary~4.25]{Gh num}, which shows that 
the ghost number of $kD_{2q}$ is $\lfloor \frac{q}{2}+1 \rfloor$.
\end{proof}

Now we consider the simple ghost number of $kD_{2ql}$.

\begin{rmk}\label{rm:D_2ql-simple-length}
Note that the only simple module in the principal block is $k$, by Lemma~\ref{le:D_2ql-even}.
Also, the inverse to the equivalence $\stmod{kD_{2q}} \to \stmod{e_0 kD_{2ql}}$
is given by restriction.
It follows that, for $M \in \stmod{e_0 kD_{2ql}}$, we have
\[ \sgl(M) = \gl(M) = \gl (M \down_{D_{2q}}). \]
\end{rmk}

To compute the simple ghost number of $kD_{2ql}$,
it remains to consider the non-principal blocks.
 From now on, we assume that $k$ contains an $l$-th primitive root of unity $\zeta$.
Let $C_{ql}$ be the cyclic subgroup of $D_{2ql}$ generated by $x$.
We will show in Proposition~\ref{pr:D_2ql-block} that
inducing up is fully-faithful on each non-principal block of $kC_{ql}$,
using the following lemmas.

We first compute the idempotent decomposition of $1$ in $kC_{ql}$.

\begin{lem}\label{le:Cql-idem}
The identity $1 \in kC_{ql}$ has an decomposition into orthogonal primitive idempotents:
\[ 1 = \sum_{i=0}^{l-1} e_i,
\text{ with } e_i = \sum_{j=0}^{l-1}(\zeta^i x^q)^j. \]
The block corresponding to $e_i$ has exactly one simple module $k_i$,
the one-dimensional module on which $x^q$ acts as $\zeta^{l-i}$.
\end{lem}

\begin{proof}
It is easy to check that the $e_i$'s are orthogonal and idempotent,
and that $e_i k_i = k_i$.
It is well known that the $k_i$'s are a complete list of simple $kC_{ql}$-modules,
so it follows that the idempotents are primitive.
\end{proof}

Since conjugation by $y$ in $D_{2ql}$ takes $e_0$ to $e_0$ and $e_i$ to $e_{l-i}$ for $i>0$,
we can deduce the idempotent decomposition for $kD_{2ql}$.

\begin{lem}\label{le:D2ql-idem}
The identity $1 \in kD_{2ql}$ has a decomposition into orthogonal primitive central
idempotents:
\[ 1 = e_0 + \sum_{i=1}^{\frac{l-1}{2}} e'_i,
\text{ with } e'_i = e_i + e_{l-i}
= \sum_{j=0}^{l-1}(\zeta^i x^q)^j
+ \sum_{j=0}^{l-1}(\zeta^{l-i} x^q)^j. \]
Moreover, the block corresponding to $e'_i$ has exactly one
simple module, namely $S_i := k_i \up_{C_{ql}} ^{D_{2ql}}$.
It follows that $\stmod{e'_i kD_{2ql}} = \thickC {D_{2ql}}{S_i}$.
\end{lem}

\begin{proof}
Clearly, the $e'_i$'s are orthogonal central idempotents.
They are primitive since there are exactly $(l+1)/2$ simple $kD_{2ql}$-modules~\cite[Theorem~3.2]{A}.
It follows that there is exactly one simple module in each block.

Define $S_i$ to be $k_i \up_{C_{ql}} ^{D_{2ql}} = kD_{2ql} \otimes_{C_{ql}} k_i$,
where $k_i$ is the simple $kC_{ql}$-module defined in Lemma~\ref{le:Cql-idem}.
With respect to the basis $\{ 1 \otimes 1,\, y \otimes 1 \}$ of $kD_{2ql} \otimes_{C_{ql}} k_i$,
it is easy to check that $S_i$ is represented by the following matrices:
\[x^l \mapsto \begin{bmatrix}
1 & 0 \\ 0 & 1 
\end{bmatrix}, \quad
x^q \mapsto \begin{bmatrix}
\zeta^{l-i} & 0 \\ 0 & \zeta^{i} 
\end{bmatrix}, \quad and \quad
y \mapsto \begin{bmatrix}
0 & 1 \\ 1 & 0 
\end{bmatrix}.
\]
And from this representation, one sees quickly that
$S_i \down_{C_{ql}} = k_i \oplus k_{l-i}$.
The action of $y$ on $S_i$ exchanges $k_i$ and $k_{l-i}$,
hence, as $kD_{2ql}$-modules,
both $k_i$ and $k_{l-i}$ generate the whole module $S_i$.
Thus $S_i$ is a simple module.
It is also clear that the module $S_i$ is in the block $e'_i kD_{2ql}$, and
so $\stmod{e'_i kD_{2ql}} = \thickC {D_{2ql}}{S_i}$.
\end{proof}

We next provide a list of all the indecomposable $kC_{ql}$-modules.
The result can be found in~\cite[p.~14, 34]{A}.
Recall that for each $1 \leq n \leq q$ there is a unique indecomposable $kC_q$-module $M_n$
of radical length $n$, and that these are all of the indecomposable $kC_q$-modules.

\begin{lem}[{\cite{A}}]\label{le:C_ql-indecomposables}
The modules $e_i(M_n \up^{C_{ql}})$, for $1 \leq n \leq q$ and $0 \leq i < l$, are
a complete list of the indecomposable $kC_{ql}$-modules. \qed
\end{lem}

Now we can show that the induction functor induces an equivalence between the non-principal blocks 
of $kC_{ql}$ and $kD_{2ql}$.

\begin{pro}\label{pr:D_2ql-block}
For $i \neq 0$, let $B_i = e_i k C_{ql}$ be a non-principal block of $k C_{ql}$.  Then
the composite of functors
\[ \stmod{B_i} \longrightarrow \stmod{k C_{ql}} \xrightarrow{\,\up_{C_{ql}}^{D_{2ql}}} \stmod{k D_{2ql}} \]
is fully-faithful,
hence induces an equivalence $\stmod{B_i} \to \stmod{e'_i kD_{2ql}}$.
\end{pro}

\begin{proof}
We begin by showing that $\up_{C_{ql}}^{D_{2ql}}$ is fully-faithful
when restricted to $\stmod{B_i}$.
Let $M := e_i(M_n \up^{C_{ql}})$ be one of the indecomposable $kC_{ql}$-modules
described in Lemma~\ref{le:C_ql-indecomposables}, and write $N := M_n \up^{C_{ql}}$.
Using Mackey's Theorem, we have
$M \up ^{D_{2ql}} \down _{C_{ql}} \iso e_i (N) \oplus y(e_i (N)) = e_i (N) \oplus e_{l-i} (N)$, and
the natural map $M \xrightarrow{\eta} M \up \down \iso e_i (N) \oplus e_{l-i} (N)$ is
an isomorphism onto $e_i (N)$.

Because $\up$ is left adjoint to $\down$, the following diagram commutes
\[
\xymatrix{
  \sHom_{C_{ql}}(M,M) \ar[r]^-{\up} \ar[rd]_{\eta_*} & \sHom_{D_{2ql}}(M\up, M\up) \ar[d]^{\iso}\\
                                                    & \ \ \sHom_{C_{ql}}(M, M\up\down).
}
\]
By the discussion in the previous paragraph, $\eta_*$ is an isomorphism,
so the horizontal map $\up$ is an isomorphism as well.
Since this is true for every indecomposable module in $\stmod{B_i}$,
it follows that the induction functor is fully-faithful when restricted to $\stmod{B_i}$, and
induces a triangulated equivalence between $\stmod{B_i}$ and its essential image.
Since $\stmod{B_i} = \thickC {C_{ql}}{k_i}$ (Lemma~\ref{le:Cql-idem}) and $k_i \up = S_i$,
the essential image of $\stmod{B_i}$ is $\stmod{e'_i kD_{2ql}} = \thickC {D_{2ql}}{S_i}$
(Lemma~\ref{le:D2ql-idem}), and
the claim follows.
\end{proof}

\begin{rmk}\label{rm:D_2ql-inverse-equivalence}
Note that the inverse of the equivalence is given by the composite of restriction 
and then projection onto the block $e_i kC_{ql}$.
\end{rmk}

We can now compute the simple ghost number of $kD_{2ql}$.

\begin{thm}\label{th:D2ql-sgn}
For $M \in \stmod{e'_i kD_{2ql}}$ with $i \neq 0$, we have
\[ \sgl(M) = \sgl(M \down_{C_{ql}}) = \gl (M \down_{C_{q}}). \]
For $M \in \stmod{e_0 kD_{2ql}}$, we have
\[ \sgl(M) = \gl (M \down_{D_{2q}}). \]
Hence the simple ghost number of $kD_{2ql} = $ the ghost number of $kD_{2ql} = \lfloor \frac{q}{2}+1 \rfloor$.
\end{thm}

\begin{proof}
We have equivalences $\stmod{B_i} \to \stmod{e'_i kD_{2ql}}$ and
$\stmod{kD_{2q}} \to \stmod{e_0 kD_{2ql}}$.
The equivalences preserve simple modules, hence radical lengths and simple ghost lengths.
Then, for $M \in \stmod{e'_i kD_{2ql}}$, we have
$\sgl(M) = \sgl(e_i(M \down_{C_{ql}})) = \sgl(e_{l-i}(M \down_{C_{ql}}))$
by Proposition~\ref{pr:D_2ql-block} and Remark~\ref{rm:D_2ql-inverse-equivalence}.
Since $M \down_{C_{ql}} = e_i(M \down_{C_{ql}}) \oplus e_{l-i}(M \down_{C_{ql}})$,
it follows that 
\[\sgl(M) = \sgl(M \down_{C_{ql}}). \]
And by Theorem~\ref{th:normal},
$ \sgl(M \down_{C_{ql}}) = \gl (M \down_{C_{q}}) .$

For $M \in \stmod{e_0 kD_{2ql}}$, we have seen in Remark~\ref{rm:D_2ql-simple-length} that
\[ \sgl(M) = \gl(M) = \gl (M \down_{D_{2q}}). \]
Since the ghost number of $kC_q$ is $\lfloor q/2 \rfloor$ (Lemma~\ref{le:C_p^r}), and the
ghost number of $kD_{2ql}$ is $\lfloor \frac{q}{2}+1 \rfloor$~\cite[Corollary~4.25]{Gh num},
it follows that the simple ghost length is maximized by $\sgl(M)$ for 
some $M \in \stmod{e_0 kD_{2ql}}$, and that
the simple ghost number of $kD_{2ql}$ equals its ghost number.
\end{proof}

\section{Groups with cyclic Sylow $p$-subgroups}\label{se:cyclic-Sylow}

We consider a group $G$ with a cyclic Sylow $p$-subgroup $P$ in this section.
When the Sylow $p$-subgroup is normal, we know from Section~\ref{ss:simple-class-normal} that
simple ghost lengths can be computed by restricting to $P$.
We show in Section~\ref{ss:normal-P} that, when $P$ is also cyclic,
the simple ghost length of a module in the principal block is equal to its ghost length 
and that
the finitely generated modules in the principal block are exactly those in $\thick k$.
We use this to compute the ghost numbers of dihedral groups at odd primes.
In Section~\ref{ss:SL(2,p)}, we study the group $SL(2,p)$ at the prime $p$,
which has a cyclic Sylow $p$-subgroup that is not normal.
Nevertheless, by restricting to the normalizer $L$ of $P$, we are able
to show that the simple generating hypothesis holds for $SL(2,p)$ for any $p$,
even though it fails for $L$ and $P$ when $p > 3$.

\subsection{The case of a cyclic normal Sylow $p$-subgroup}\label{ss:normal-P}

Let $k$ be a field of characteristic $p$, and let $G$ be a finite group with
cyclic Sylow $p$-subgroup $P = C_{p^r}$.
We assume that $k$ is algebraically closed and that $C_{p^r}$ is normal in $G$.

Since $P \leq G$ is normal, Theorem~\ref{th:normal} says that
\[ \sgl(M) = \gl ( M\down_P ) \]
for $M \in \stmod{kG}$.
In this section, using that $P$ is in addition cyclic, we are going to show that
\[ \sgl(M) = \gl (M) \]
for $M \in \stmod{B_0}$, 
as we found for direct products in Proposition~\ref{pr:direct product}.

Our approach is as follows.
We will show that all simple modules in the principal block $\StMod{B_0}$ are suspensions of the trivial module $k$.
Hence the simple ghost projective class and the ghost projective class
coincide when both are pulled back to $\StMod{B_0}$.
It then follows that $\thick k$ equals $\stmod{B_0}$ and that
for $M$ in $\stmod{B_0}$, its ghost length is the same
as its simple ghost length.

We say that a $kG$-module $M$ is \dfn{uniserial} if
the successive quotients in the radical sequence associated to $M$ are simple.
Note that this is equivalent to 
the successive quotients in the socle sequence associated to $M$ being simple.

An important fact about the representations of $G$ when its Sylow $p$-subgroup is normal and cyclic is that
the indecomposable modules are uniserial:

\begin{thm}[{\cite[pp.~42--43]{A}}]\label{th:uniserial}
Let $G$ be a finite group, and let $k$ be a field of characteristic $p$.
Assume that the Sylow $p$-subgroup $P$ of $G$ is normal and cyclic.
Then there are finitely many indecomposable $kG$-modules.
Every indecomposable module $M$ is uniserial and is characterised by
its radical length and the simple module $M/ \rad(M)$. \qed
\end{thm}

Recall that in general 
there is a bijection between indecomposable projective $kG$-modules
and simple $kG$-modules given by the assignment
that sends a projective module $Q$ to its radical quotient $Q/ \rad(Q)${~\cite[Theorem~5.3]{A}}.
The inverse sends a simple module to its projective cover, i.e., the unique
indecomposable projective module that surjects onto it.
Also note that for a projective $kG$-module $Q$,
we have an isomorphism $Q/ \rad(Q) \iso \soc(Q)${~\cite[Theorem~6.6]{A}}.

When $P = C_{p^r} \leq G$ is cyclic and normal, we can say more.

\begin{lem}
Let $G$ be a finite group with cyclic normal Sylow $p$-subgroup $P = C_{p^r}$,
let $k$ be a field of characteristic $p$,
and let $Q$ be the projective cover of the trivial module $k$.
If $S$ is a simple module, then $Q \otimes S$ is its projective cover.
\end{lem}

The proof we give here uses the idea in~\cite[p.~35]{A},
where it is shown that the projective cover of a simple module $S$ restricts to
a free $kC_{p^r}$-module of rank $\dim(S)$.

\begin{proof}
We first show that the restriction of $Q$ to $C_{p^r}$ is free of rank $1$.
Indeed, $Q \down_{kC_{p^r}}$ is projective, hence free, since $C_{p^r}$ is a $p$-group, and
the rank of $Q \down_{kC_{p^r}}$ is equal to the dimension of its socle.
But $C_{p^r}$ is normal in $G$, so,
by Lemma~\ref{le:res-simple},
$\soc(Q \down_{kC_{p^r}}) = \soc(Q) \down_{kC_{p^r}} = k$, and
the claim follows.

Since $Q \otimes M$ is projective for any $kG$-module $M$~\cite[Lemma~7.4]{A},
the module $Q \otimes S$ is projective, and
its restriction to $C_{p^r}$ is free.
By comparing dimensions, one sees that the restriction of $Q \otimes S$
is free of rank $\dim(S)$, and
by Lemma~\ref{le:res-simple} again,
the dimension of $\soc(Q \otimes S)$ is also equal to $\dim S$.
Since $S \iso k \otimes S$ is contained in $\soc(Q \otimes S)$,
we deduce that $S \iso \soc(Q \otimes S)$.
It follows that $Q \otimes S$ is indecomposable,
hence is the projective cover of $S$.
\end{proof}

We continue to write $Q$ for the projective cover of the trivial module $k$.
Since $Q\down_{kC_{p^r}} \iso kC_{p^r}$, the radical layers of $Q$ all have dimension $1$.
Now, let $W$ be the 1-dimensional simple module $\rad(Q) / \rad^2(Q)$.
We show that $W \iso \widetilde{\Omega}^2 k$.

\begin{lem}\label{le:Wn}
The module $W$ is isomorphic to the double desuspension of the trivial module $k$,
i.e., $W \iso\widetilde{\Omega}^2 k$.
Moreover, each composition factor of $Q$ is a tensor power $W^{\otimes n}$ of the module $W$.
\end{lem}

\begin{proof}
The map $Q \otimes W \to W$ factors through the quotient map $\pi: \rad(Q) \to W$
and gives a map $f: Q \otimes W \to \rad(Q)$.
Since $\ker(\pi) = \rad^2(Q)$ is the unique maximal submodule of $\rad(Q)$,
the map $f$ is surjective.
As we saw for $Q$, the radical layers of $Q \otimes W$ are also all 1-dimensional,
so $Q$ and $Q \otimes W$ have the same radical length.
Since the radical length of $\rad(Q)$ is one less than that of $Q \otimes W$,
the composite
\[ W \lra Q \otimes W \llra{f} \rad(Q)\]
is zero, where the map $W \to Q \otimes W$
is the inclusion of the last radical (which equals the socle) of $Q \otimes W$.
By comparing dimensions, one sees that this is a short exact sequence.
And since $\rad(Q) \iso \widetilde{\Omega} k$,
we have $W \iso \widetilde{\Omega}^2 k$.

To see that the composition factors of $Q$ are $W^{\otimes n}$,
first note that
\[ \rad^n (Q) / \rad^{n+1} (Q) \iso 
   \rad^{n-1} (Q \otimes W) / \rad^n (Q \otimes W)\]
for $1 \leq n \leq p^r-1$.
We get these isomorphisms by comparing the radical layers
along the surjective map $f: Q \otimes W \to \rad(Q)$,
using that both $Q$ and $Q \otimes W$ have 1-dimensional layers.
On the other hand,
$\rad^n (Q \otimes W) \iso \rad^n (Q) \otimes W$,
since tensoring with $W$ preserves the radical layers.
Thus
\[ \rad^{n-1} (Q \otimes W) / \rad^n (Q \otimes W) \iso
   (\rad^{n-1} (Q) / \rad^n (Q)) \otimes W. \]
Combining the two displayed isomorphisms and using that $W = \rad(Q)/\rad^2(Q)$, 
it follows by induction that
\[ \rad^n (Q) / \rad^{n+1} (Q) \iso W^{\otimes n}. \qedhere\]
\end{proof}

Note that $M \otimes W \iso \widetilde{\Omega}^2 M$ for any module $M$.
In particular, $W^{\otimes n} \iso \widetilde{\Omega}^{2n} k$.
Also note that, more generally,
the indecomposable projective module $Q \otimes S$ is uniserial
with composition factors $W^{\otimes n} \otimes S$.
Thus the following lemma, together with Lemma~\ref{le:Wn},
implies that the modules $W^{\otimes n}$ are all the simple modules in $\StMod{B_0}$.
(See also~\cite[Exercise~13.3]{A}.)
Thus the simple modules in $\StMod{B_0}$ are all in $\cF$,
and so simple ghosts and ghosts agree in the principal block.

\begin{lem}[{\cite[Proposition 13.3]{A}}]
Let $k$ be a field of characteristic $p$, and
let $G$ be a finite group with a cyclic normal Sylow $p$-subgroup $C_{p^r}$.
Then two simple modules $S$ and $T$ are in the same block if and only if
there exists a sequence of simple modules
\[ S = S_1, S_2, \ldots, S_m = T \]
such that $S_i$ and $S_{i+1}$ are composition factors of an indecomposable 
projective $kG$-module, for $1 \leq i < m$. \qed
\end{lem}

We can use the above observations to compute the ghost number of $kG$.

\begin{thm}\label{th:cyclic-Sylow}
Let $k$ be a field of characteristic $p$, and
let $G$ be a finite group with a cyclic normal Sylow $p$-subgroup $C_{p^r}$.
Then $\thickC G k = \stmod{B_0}$, and
a map in $\thickC G k$ is a ghost if and only if
its restriction to $\stmod{kC_{p^r}}$ is a ghost.
As a result,
\[
  \text{ghost number of }kG = \text{ghost number of }kC_{p^r} = \lfloor p^r/2 \rfloor .
\] 
Moreover, let $M$ be a uniserial $kG$-module of radical length $l$ in $\thickC G k$.
Then
\[
\gl(M) = \sgl(M) = \min(l,\, p^r - l).
\]
In particular, using the natural terminology, 
the ghost number of $kG$ is equal to the simple ghost number of $B_0$.
\end{thm}

\begin{proof}
Since the simple modules in the principal block are contained in $\cF$,
the pullback of the simple ghost projective class to $\StMod{B_0}$ coincides with
the pullback of the ghost projective class to $\StMod{B_0}$.
It follows that $\thickC G k = \stmod{B_0}$, and $\gl(M) = \sgl(M)$ for a module $M$ in $\stmod{B_0}$.
Since $P \leq G$ is normal, $\sgl(M) = \gl(M \down_P)$ for $M \in \stmod{kG}$,
by Theorem~\ref{th:normal}.
Hence
\[ \gl(M) = \sgl(M) = \gl(M \down_P), \]
and we can compute the ghost lengths in $kG$ by restricting to $kC_{p^r}$.
The ghost lengths in $kC_{p^r}$ are computed in~\cite{Gh in rep}
(summarized in Lemma~\ref{le:C_p^r} below).
\end{proof}

\begin{rmk}\label{rm:module_W}
We give a concrete description of the module $W$~\cite[Exercise~5.3]{A}.
Let $x$ be a generator of the cyclic group $C_{p^r}$.
Then the one-dimensional module $W$ is given by the group homomorphism that
sends $g \in G$ to $\overline{\alpha(g)} \in k^{\times}$, where
$\alpha(g)$ is the integer such that $gxg^{-1}=x^{\alpha(g)}$ and
$\overline{\alpha(g)}$ is its image under the canonical map $\Z \to k$.
If we further compose this map with the self map on $k^{\times}$ that
takes $\alpha$ to $\alpha^n$, we get the module $W^{\otimes n}$.
Since $\overline{\alpha(g)}$ lands in $\F_p \subseteq k$,
we always have $W^{\otimes (p-1)} = k$.
\end{rmk}

Let $M$ be a non-projective uniserial module with radical length $l \geq 2$.
We give an explicit construction of a (weakly) universal simple ghost out of $M$.
Let $W^*$ be the dual of $W$, so $W \otimes W^* \iso k$.
We have
\begin{equation}\label{eq:W*}
M / \rad(M) \iso (\rad(M )/ \rad^2(M)) \otimes W^* \iso \rad(M \otimes W^*)/ \rad^2(M \otimes W^*).
\end{equation}
To see that the first isomorphism holds, note that it holds for the module $Q$,
hence for the modules $Q \otimes S$ with $S$ simple.
Since $M$ is a quotient of one of the uniserial modules $Q \otimes S$,
the isomorphism holds for $M$ too.
Recall by Theorem~\ref{th:normal} that a map $f$ is a simple ghost if and only if
its restriction to a Sylow $p$-subgroup is a ghost.
And for a $p$-group $P$, we know that a ghost $g: M \to N$ has
$\im (g) \subseteq \rad(N)$ and $\soc(M) \subseteq \ker (g)$ by~\cite[Corollary~2.6]{Gh in rep}.
Hence we consider the short exact sequences
\[ 0 \lra \soc(M) \lra M \llra{\pi} M/ \soc(M) \lra 0\]
and
\[ 0 \lra \rad(M \otimes W^*) \llra{i} M \otimes W^* \lra M \otimes W^* / \rad(M \otimes W^*) \lra 0. \]
Equation~\eqref{eq:W*} implies that $M/ \soc(M) \iso \rad(M \otimes W^*)$.
Now let $g$ be the composite
$M \xrightarrow{\pi} M/ \soc(M) \iso \rad(M \otimes W^*) \xrightarrow{i} M \otimes W^*$.
Then $\im (g) \subseteq \rad(N)$ and $\soc(M) \subseteq \ker (g)$.
By Lemma~\ref{le:res-simple}, the inclusions still hold when restricted to
the normal Sylow $p$-subgroup $C_{p^r}$.
Since $\Omega^2 k \iso k$ in $\stmod{kC_{p^r}}$,
the proof of~\cite[Proposition~2.1]{Gh in rep} shows that $g \down_{C_{p^r}}$ is a ghost.
So by Theorem~\ref{th:normal}, the map $g$ is a simple ghost.
One can check that the fibre of $g$ is $\soc(M) \oplus \Omega(M \otimes W^* / \rad(M \otimes W^*))$.
Thus $g$ is a weakly universal simple ghost.
This process can be iterated, producing composites $M \to M \otimes (W^*)^n$ of $n$ simple ghosts
which are nonzero for $n < \sgl(M)$. 
If $M$ is in the principal block, then these simple ghosts are ghosts,
and so we have exhibited the ghosts predicted by Theorem~\ref{th:cyclic-Sylow}.

\begin{thm}\label{th:dihedral-odd}
Let $D_{2ql}$ be a dihedral group, with $q$ a power of $2$ and $l$ odd.
Let $k$ be a field of characteristic $p$ which divides $2ql$.
If $p$ is odd, then the ghost number of $kD_{2ql}$ is 
$\lfloor p^r/2 \rfloor$, where $p^r$ is the $p$-primary part of $l$.
If $p$ is even, then the ghost number of $kD_{2ql}$ is
$\lfloor q/2 + 1 \rfloor$.
\end{thm}

\begin{proof}
If $p$ is odd, then its Sylow $p$-group is cyclic and normal, 
so its ghost number is given by Theorem~\ref{th:cyclic-Sylow}.
If $p$ is even, then its ghost number was computed in Corollary~\ref{co:D_2ql-even}.
\end{proof}

\subsection{The simple generating hypothesis for the group $SL(2,p)$}\label{ss:SL(2,p)}

In this section, we show that the simple generating hypothesis holds for $kG$,
where $G$ is the group $SL(2,p)$ of order $p(p-1)(p+1)$ and $k$ is a field of characteristic $p$.
Background on representations of $SL(2,p)$ can be found in~\cite[p.~14, 75]{A}.
We will also need to know about representations of the normalizer $N(P)$ of $P$ in $SL(2,p)$,
which illustrates the results of Section~\ref{ss:normal-P}.

We let $P \leq G$ consist of all elements of the form
$\begin{pmatrix}
1 & 0 \\ c & 1
\end{pmatrix}$.
$P$ has order $p$ and is a Sylow $p$-subgroup of $G$.
Let $L = N(P)$ be the normalizer of $P$ in $G$.
It consists of the elements of the form
$\begin{pmatrix}
a & 0 \\ c' & 1/a
\end{pmatrix}$.

For $i \in \Z$, consider the one-dimensional simple module $S_i$ of $L$ given by the group map
$L \to k^{\times}$ that sends
$\begin{pmatrix}
a & 0 \\ c' & 1/a
\end{pmatrix}$ to $a^i$.
Note that $S_0 = k$ is the trivial representation.
Clearly, $S_i \iso S_j$ if and only if $i \equiv j$ (modulo $p-1$) and
$S_i \otimes S_j \iso S_{i+j}$.
These are all of the simple $kL$-modules,
since there can be at most $p-1$ non-isomorphic indecomposable projective $kL$-modules.

Applying the discussion in Section~\ref{ss:normal-P} to the group $L$,
one obtains a $kL$-module $W \iso \widetilde{\Omega}^2 k$.  
By Remark~\ref{rm:module_W}, one can check that $W \iso S_{-2}$.
It follows that $kL$ has two blocks,
with the module $S_i$ in the principal block if and only if $i$ is even.
Moreover, $S_{-2i} \iso W^{\otimes i} \iso \widetilde{\Omega}^{2i} k$, using Lemma~\ref{le:Wn},
so all of the simple modules in the principal block are suspensions of the trivial module $k$.
By Theorems~\ref{th:normal} and~\ref{th:cyclic-Sylow}, the simple ghost number of $kL$,
the ghost number of $kL$ and the ghost number of $kP$ are all equal
to $\lfloor p/2 \rfloor$.

We will show below that the simple ghost number of $kG$ is actually $1$,
which is surprising since the simple generating hypothesis fails for its subgroups $P$ and $L$
when $p > 3$.
It is even more surprising in view of the next result, which shows that
$\stmod{kG}$ and $\stmod{kL}$ are equivalent.

The Sylow $p$-subgroup $P$ is cyclic of order $p$.
Thus it is a trivial intersection subgroup of $G$ (i.e., $gPg^{-1} \cap P$ is either $P$ or trivial), and
we have an equivalence between $\stmod{kG}$ and $\stmod{kL}$ by restriction and inducing up.
Moreover, since the equivalence preserves the trivial representation both ways,
the ghost number of $kG$ equals that of $kL$:

\begin{thm}[{\cite[Theorems~10.1, 10.3]{A}}]\label{th:trivial-intersection}
Let $G$ be a finite group, and let $k$ be a field whose characteristic divides the order of $G$.
Let $P$ be a Sylow $p$-subgroup of $G$ and let $L = N(P)$ be the normalizer of $P$ in $G$.
Assume that $P$ is a trivial intersection subgroup of $G$.  Then 
the restriction functor
\[\stmod{kG} \lra \stmod{kL}\]
is an equivalence, with inverse given by the inducing up functor.
In particular, the ghost number of $kG$ is equal to the ghost number of $kL$,
namely $\lfloor p/2 \rfloor$.\qed
\end{thm}

Note that the equivalence does not preserve simple modules.
By Theorem~\ref{th:trivial-intersection},
to study the simple ghost number of $\StMod{kG}$,
it is equivalent to study the pullback projective class of $(\cS, \sG)$ in $\StMod{kL}$,
i.e., the projective class in $\StMod{kL}$ generated by the modules $S \down_L$,
for $S$ a simple $kG$-module.
We are going to show that this projective class contains all finitely-generated modules, and
it will follow that the simple generating hypothesis holds for $kG$.

\begin{thm}\label{th:SL2p}
Let $G = SL(2,p)$.
Every module in $\stmod{kG}$ is a direct sum of suspensions of simple modules.
In particular, the simple generating hypothesis holds for $kG$.
\end{thm}

Note that despite the equivalence of Theorem~\ref{th:trivial-intersection},
we already observed that the simple generating hypothesis does not hold for $kL$
unless $p$ is $2$ or $3$.

\begin{proof}
By the remarks immediately preceding the theorem, it suffices to show that the modules $S \down _L$,
with $S$ a simple module in $\stmod{kG}$,
generate everything in $\stmod{kL}$ under direct sums, suspensions and retracts.

By Theorem~\ref{th:uniserial}, the indecomposable $kL$-modules are $M_{i,j}$,
for $1 \leq i \leq p$ and $0 \leq j \leq p-2$,
where $M_{i,j}$ has radical length $i$ and radical quotient $M/\rad(M) \iso S_j$.
It thus suffices to show that each module $M_{i,j}$ is a suspension of some $S \down_L$.
For convenience, in the following we will interpret the subscript $j$ modulo $p-1$.

There are $p$ simple $kG$-modules~\cite[p.~14]{A}, and
we write $V_1, \dots, V_p$ for their restrictions to $L$.
The $kL$-module $V_i$ is uniserial of radical length $i$,
with radical quotient $V_i / \rad(V_i) \iso S_{i-1}$~\cite[p.~76]{A},
so $V_i = M_{i,i-1}$.
Note that the module $V_1$ is trivial and the module $V_p$ is projective.
The case $p = 2$ follows immediately, since $L = C_2$, and
$M_{1,0} = V_1 \iso k$ and $M_{2,0} = M_{2,1} = V_2 \iso kC_2$ are the only two
indecomposable $kL$-modules.
Thus we assume that $p$ is odd.

Recall that $W \iso \widetilde{\Omega}^2(k)$, hence $\blank \otimes W$ is isomorphic to
the functor $\Omega^2(\blank)$ on $\stmod{kL}$.
Since \mbox{$\blank \otimes W$} preserves radical lengths
and shifts the simple module $S_j$ to $S_{j-2}$,
we have a stable isomorphism $\Omega^{2k} V_i \iso M_{i,i-1-2k}$ for $k \in \Z$.
This gives all modules $M_{i,j}$ where $i+j$ is odd.

To get the modules $M_{i,j}$ with $i+j$ even and $1 \leq i < p$, note that
$V_p \otimes S_{p-i-1}$ is the projective cover of $V_{p-i}$. 
It follows that $\widetilde{\Omega} V_{p-i}$ has radical length $i$ and
radical quotient $S_{i-2}$, i.e.,
$\widetilde{\Omega} V_{p-i} \iso M_{i,i-2}$.
Then we can apply $\Omega^{2k}$ again to obtain the modules $M_{i,j}$ where $i+j$ is even.
\end{proof}

In general, for which groups the simple generating hypothesis holds remains open.

\section{Strong ghosts}\label{se:strong-ghosts}

In Section~\ref{ss:sgh class}, we motivate and define strong ghosts and show that
the strong ghost number of a group algebra $kG$ equals the strong ghost number of $kP$, 
where $P$ is a Sylow $p$-subgroup of $G$.
In Section~\ref{ss:sgn-C_p^r}, we compute the strong ghost number of
any cyclic $p$-group over a field of characteristic $p$.
In Section~\ref{ss:sgn-D_4q}, we show that the strong ghost number of
any dihedral $2$-group over a field of characteristic $2$
is between $2$ and $3$.

\subsection{The strong ghost projective class}\label{ss:sgh class}

If $H$ is a subgroup of a finite group $G$, then it is rare for
the restriction functor from $G$ to $H$ to preserve ghosts.
For example, we saw in Section~\ref{ss:A_4} that restriction
from the group $A_4$ to its Sylow $p$-subgroup $P$
does not preserve ghosts.
As another example, if $G$ is a $p$-group and $N \leq G$ is any normal subgroup,
then the restriction from $G$ to $N$ does not preserve ghosts,
since $k\up_N^G$ is indecomposable~\cite[Theorem~8.8]{A} and
is not a suspension of $k$.
Strong ghosts, which were introduced in~\cite{sgh}, will by definition
restrict to ghosts.

\begin{de}
Let $G$ be a finite group, and let $k$ be a field whose characteristic divides the order of $G$.
A map in $\StMod{kG}$ is called a \dfn{strong ghost} if 
its restriction to $\StMod{kH}$ is a ghost for every subgroup $H$ of $G$.
\end{de}

It follows immediately that the restriction of a strong ghost to any subgroup is again a strong ghost.

In~\cite{sgh}, Carlson, Chebolu and Min\'a\v{c} study strong ghosts in $\thick k$,
but their results imply the following theorem, 
which says that most groups admit strong ghosts in $\stmod{kG}$:

\begin{thm}[Carlson, Chebolu and Min\'a\v{c} \cite{sgh}]\label{th:stgh}
Let $G$ be a finite group, and let $k$ be a field whose characteristic divides the order of $G$. 
Then every strong ghost in $\stmod{kG}$ is stably trivial if and only if 
the Sylow $p$-subgroup of $G$ is $C_2$, $C_3$, or $C_4$.
\end{thm}

Note that in passing from ghosts to strong ghosts,
we only get one more $p$-group, namely $C_4$,
where all strong ghosts are stably trivial. 

We next observe that strong ghosts form an ideal of a projective class
and use this in further study of strong ghosts.

Let $H$ be a subgroup of $G$. We know that the restriction functor 
\[
 \down_H \colon \StMod{kG} \lra \StMod{kH}
\]
is both left and right adjoint to the induction functor
\[
 \up^G: \StMod{kH} \lra \StMod{kG}.
\]
The pullback (see Definition~\ref{de:pullback}) of the ghost projective class along the restriction functor consists of
maps in $\StMod{kG}$ which restrict to ghosts in $\StMod{kH}$.
The intersection of such ideals when $H$ ranges over all subgroups of $G$
consists of exactly the strong ghosts and
again forms an ideal of a projective class:
the relative projectives are obtained from modules of the form $k \up_H^G$
by closing under suspensions, desuspensions, direct sums and retracts.
This is the \dfn{strong ghost projective class} in $\StMod{kG}$ and
is denoted by $(\stF, \stG)$.
(In the terminology of~\cite{Chr}, it is the \emph{meet} of the pullbacks.)

Note that we can set $\bP = \{ k\up_H^G \st H \text{ is a subgroup of } G\}$ in $\StMod{kG}$, and
this generates exactly the strong ghost projective class.
Since every $kG$-module $M$ is a summand of $M \down_P \up_P^G$,
where $P$ is a Sylow $p$-subgroup of $G$,
and induction is a triangulated functor,
we have that $\thickC G \bP = \stmod{kG}$.
Hence, using the terminology in Section~\ref{ss:length-general},
Theorem~\ref{th:stgh} is the statement that
the generating hypothesis with respect to $\bP$ holds in $\StMod{kG}$ if and only if
the Sylow $p$-subgroup of $G$ is $C_2$, $C_3$, or $C_4$.

For $M \in \stmod{kG}$, we define the \dfn{strong ghost length} of $M$,
denoted by $\stgl (M)$,
to be the Freyd length of $M$ with respect to $\bP$, i.e.,
$\stgl(M) = \Flen_{\bP} (M)$.
The \dfn{strong ghost number} of $kG$ is defined to be the Freyd number of
$\StMod{kG}$ with respect to $\bP$.

One can show that
strong ghosts induce up to strong ghosts by proving the dual statement,
i.e., that relative projectives restrict to relative projectives.
This follows from Mackey's Theorem (Theorem~\ref{th:mackey}) and
the observation that $s(\Omega_H^n k) \iso \Omega_{sHs^{-1}}^nk$~\cite{sgh}.
Since the induction functor is always faithful, one obtains
the following result:

\begin{pro}[Carlson, Chebolu and Min\'a\v{c}~\cite{sgh}]\label{pr:ccm-sgh}
Let $G$ be a finite group, and let $k$ be a field whose characteristic
divides the order of $G$. Let $H$ be a subgroup of $G$.
If $g$ is a stably non-trivial strong ghost in $\StMod{kH}$,
then $g \up^G$ is a stably non-trivial strong ghost in $\StMod{kG}$.\qed
\end{pro}

Next, we prove that the induction functor preserves strong ghost lengths.

\begin{pro}\label{pr:stgl}
Let $G$ be a finite group, and let $k$ be a field whose characteristic
divides the order of $G$.
Let $H$ be a subgroup of $G$.
Then for any $M$ in $\stmod{kH}$, $\stgl(M \up^G) = \stgl(M)$,
and so the strong ghost number of $kG$ is at least as big as the
strong ghost number of $kH$.
Moreover, if $P$ is a Sylow $p$-subgroup of $G$, then
\[
\text{strong ghost number of } kP = \text{strong ghost number of } kG.
\]
\end{pro}

\begin{proof}
The proof is essentially the same as the proof of Theorem~\ref{th:normal}.
By Proposition~\ref{pr:ccm-sgh}, we have $\stgl(M \up^G) \geq \stgl(M)$.
Conversely, since the natural isomorphism $\alpha: \sHom_G(M \up^G, L) \to \sHom_H(M, L \down_H)$
preserves strong ghosts, $\stgl(M \up^G) \leq \stgl(M)$.

When $P$ is a Sylow $p$-subgroup of $G$, the restriction functor is faithful by
Lemma~\ref{le:background}(i), and
the last equality follows easily.
\end{proof}

\subsection{Strong ghost numbers of cyclic $p$-groups}\label{ss:sgn-C_p^r}

We study the strong ghost numbers of cyclic $p$-groups in this section.
Our result suggests that the notion of a strong ghost is much stronger than that of a ghost.

We first review ghost lengths in $\stmod{kC_{p^r}}$,
following~\cite[Section~5.1]{Gh in rep}.

\begin{lem}\label{le:C_p^r}
Let $G = C_{p^r}$ be a cyclic group of order $p^r$ with generator $g$,
let $k$ be a field of characteristic $p$, and
let $M_n$ be the indecomposable $kC_{p^r}$-module of radical length $n$.
Then the self map $g-1$ on $M_n$ is a weakly universal ghost, 
i.e., any ghost with domain $M_n$ factors through $g-1$.
Moreover $\gl (M_n) = \min (n, p^r-n)$ and the ghost number of
$kG$ is $\lfloor p^r/2 \rfloor$.
\end{lem}

\begin{proof}
That the map $g-1$ is a ghost is proved in~\cite[Lemma~2.2]{GH for p}.
It is weakly universal, since it fits into a triangle
\[ k \oplus \Sigma k \lra M_n \xrightarrow{g-1} M_n \lra k \oplus \Sigma k. \]
The $l$-fold composite $(g-1)^l$ on $M_n$ is stably trivial if and only if
$l \geq \min(n, p^r-n)$ (see~\cite[Propositions~5.2, 5.3]{Gh in rep}).
Hence $\gl (M_n) = \min (n, p^r-n)$.
Since all indecomposables are of this form, the ghost number of
$kG$ is $\lfloor p^r/2 \rfloor$.
\end{proof}

\begin{thm}\label{th:C_p^r}
Let $G = C_{p^r}$ be a cyclic group of order $p^r$,
let $k$ be a field of characteristic $p$, and
let $M_n$ be the indecomposable $kC_{p^r}$-module of radical length $n < p^r$.
Writing $N = \min(n, p^r - n) = \gl(M_n)$,
we have the following:
\begin{enumerate}[(i)]
\item\label{item:i}
If $N \leq p^{r-1}$, then
\[ \stgl( M_n)
= \begin{cases} 1 & \text{if } N \mid p^r,\\
                \displaystyle 2 & else.
  \end{cases}\]
\item\label{item:ii}
If $N > p^{r-1}$, then
\[ \stgl( M_n) = \left\lceil \frac{N}{p^{r-1}} \right\rceil = \left\lceil \frac{\gl(M_n)}{p^{r-1}} \right\rceil . \]
\end{enumerate}
It follows that 
\[
\text{strong ghost number of } kG = \begin{cases}
  \bigg\lceil \displaystyle \frac{p+1}{2} \bigg\rceil , & \text{when $p=2$ and $r \geq 3$, or $p$ is odd and $r \geq 2$}\\[12pt]

  \bigg\lceil \displaystyle \frac{p-1}{2} \bigg\rceil , & \text{otherwise}
\end{cases}
\]
\end{thm}

\begin{proof}
We divide the proof into three cases:
\smallskip

\noindent
\textbf{Case 1:} We first determine the indecomposable modules in $\stF$,
i.e., those of strong ghost length $1$.
The set $\stF$ is generated by $\bP = \{ k\up_{C_{p^j}}^{C_{p^r}} = M_{p^{r-j}} \st 1 \leq j \leq r \}$, and
so an indecomposable module $M_n$ is in $\bP$ if and only if $n \mid p^r$.
Since $\stF$ also contains the suspensions of modules in $\bP$
and $\Sigma M_n \iso M_{p^r - n}$, it follows that
$\stgl(M_n) = 1$ if and only if $n \mid p^r$ or $(p^r - n) \mid p^r$,
i.e., $N \mid  p^r$.

This implies that $\bP \subseteq \cF_{p^{r-1}}$, or equivalently,
that $\cG^{p^{r-1}} \subseteq \stG$, which will be useful below.
\smallskip

\noindent
\textbf{Case 2:}
For $N < p^{r-1}$,
we show that $M_n$ is contained in $\stF_2$.
Indeed, for such $n$ we have a triangle
\[ M_n \oplus \Sigma M_n \lra M_{p^{r-1}} \xrightarrow{(g-1)^N} M_{p^{r-1}} \lra M_n \oplus \Sigma M_n , \]
where $g$ is a generator of $C_{p^r}$.
Hence $M_n \in \stF_2$ and $\stgl(M_n) \leq 2$, completing the proof
of~\eqref{item:i}.
\smallskip

\noindent
\textbf{Case 3:}
We compute the strong ghost length of $M_n$ for $N > p^{r-1}$.
By the previous observation, the self map $(g-1)^{p^{r-1}}$ on $M_n$ is a strong ghost.
This map fits into the triangle
\[ M_{p^{r-1}} \oplus \Sigma M_{p^{r-1}} \lra M_n \xrightarrow{(g-1)^{p^{r-1}}} M_n \lra M_{p^{r-1}} \oplus \Sigma M_{p^{r-1}}, \]
with fibre in $\stF$, so
it is a weakly universal strong ghost.
By Lemma~\ref{le:C_p^r}, its $j$th power is stably trivial
if and only if $j p^{r-1} \geq N = \gl(M_n)$.
The equality in~\eqref{item:ii} then follows.

\smallskip

The calculation of the strong ghost number follows from these results:

When $p=2$, the ghost number of $kC_{2^r}$ is $2^{r-1}$, hence
all $C_{2^r}$-modules are dealt with in~\eqref{item:i}, and
the strong ghost number of $kC_{2^r}$ is $2$ provided $r \geq 3$,
and $1$ otherwise.

When $p$ is odd, the modules in~\eqref{item:ii} dominate.  The strong ghost length
is maximized when $N = (p^r-1)/2$ (the ghost number of $kC_{p^r}$) and is
\[
  \left\lceil \frac{p^r-1}{2 p^{r-1}} \right\rceil = \left\lceil \frac{p-\frac{1}{p^{r-1}}}{2} \right\rceil ,
\]
which simplifies to the desired expressions.
\end{proof}

\subsection{Strong ghost numbers of dihedral 2-groups}\label{ss:sgn-D_4q}

In this section we find an upper bound for the strong ghost number of
a dihedral $2$-group, using the result from the previous section on
the strong ghost numbers of cyclic $p$-groups.

Let $k$ be a field of characteristic $2$.
We write $D_{4q}$ for the dihedral $2$-group of order $4q$, with $q$ a power of $2$:
\[ D_{4q} = \langle x,y \st x^2=y^2=1,\, (xy)^q=(yx)^q \rangle . \]
It has a normal cyclic subgroup $C_{2q}$, generated by $g = xy$. 
We prove the following theorem on the strong ghost number of $kD_{4q}$:

\begin{thm}\label{th:D_4q}
Let $k$ be a field of characteristic $2$, and
let $D_{4q}$ be the dihedral $2$-group of order $4q$,
with $q = 2^r$ and $r \geq 1$. Then
\[ 2 \leq \text{the strong ghost number of } kD_{4q} \leq 3. \]
\end{thm}

Recall that the strong generating hypothesis fails for $kD_{4q}$ by Theorem~\ref{th:stgh},
so the strong ghost number of $kD_{4q}$ is at least 2 for $r \geq 0$.
When $r = 0$, $D_{4q} \iso C_2 \times C_2$, and the strong ghost number of $k(C_2 \times C_2)$ is $2$.
Indeed, for a $p$-group, the strong ghost number is bounded above by the ghost number, and
by~\cite[Corollary~5.13]{Gh in rep}, the ghost number of $k(C_2 \times C_2)$ is also $2$.

Our goal will be to prove the upper bound.
We will make use of the notation from~\cite{Benson} (see also~\cite[Section~4.6]{Gh num}),
where the indecomposable $kD_{4q}$-modules are written using words in the letters $a$ and $b$.
By the proof of~\cite[Theorem~4.24]{Gh num}, every non-projective indecomposable
$kD_{4q}$-module $M$ sits in a triangle
$\Omega W \to M \to M'' \to W$, where $M''$ is a sum of modules of the form $M((ab)^s)$ and $M((ab)^s a)$,
for $0 \leq s < q$, and the same, with $a$ and $b$ exchanged,
and $W$ is a sum of suspensions and desuspensions of the trivial module.
Thus, by Lemma~\ref{le:Flengths-in-triangle},
it will suffice to show that the modules $M((ab)^s)$ and $M((ab)^s a)$ have
strong ghost length at most $2$.

\begin{proof}[Proof of Theorem]

By the discussion above, it suffices to show that
\[ \stgl (M((ab)^s)) \leq 2 \quad\text{and}\quad \stgl (M((ab)^s a)) \leq 2 \]
for $0 \leq s < q$.

It will be convenient to make the following notational convention:
when we write $(ab)^\frac{m}{2}$, we mean $a b a \cdots$ with $m$ letters in total.
For example, $(ab)^\frac{5}{2} = a b a b a$.
In addition, $(ba)^{-\frac{m}{2}}$ denotes $((ba)^\frac{m}{2})^{-1}$, so
$(ba)^{-\frac{5}{2}} = b^{-1} a^{-1} b^{-1} a^{-1} b^{-1}$.
Let $M = M ( (ab)^{\frac{q}{2}} (ba)^{-\frac{q}{2}}, \id) \iso k \up_{C_2}^{D_{4q}} \iso M_q \up_{C_{2q}}^{D_{4q}}$,
which has strong ghost number $1$.
Similarly, for $0 \leq m \leq q-1$, we write
$M_m'$ for the module 
$M( (ab)^{\frac{m}{2}} (ba)^{-\frac{m}{2}}, \id)$,
which has $2m$ letters in total.
Then $M_m' \iso M_m \up_{C_{2q}}^{D_{4q}}$,
as one can check that $(1-xy)^m(yxyx \cdots) = XYXY \cdots - YXYX \cdots$ 
($m$ factors in each expression) by induction,
where $X = x-1$ and $Y = y-1$.
Thus, by Proposition~\ref{pr:stgl}, $M_m'$ has strong ghost length at most $2$, since $M_m$ does.
Inducing up the triangle
$M_q \xrightarrow{\!(1-g)^m\!\!} M_q \to M_m \oplus M_{2q-m} \to M_q$ in $\stmod{kC_{2q}}$,
we get the triangle
\begin{equation}\label{eq:(1-g)m}
M \xrightarrow{(1-xy)^m} M \xrightarrow{\ \alpha\ } M_m' \oplus M_{2q-m}' \xrightarrow{\ \beta\ } M.
\end{equation}
Let $j: M_m' \oplus M_{2q-m}' \to k$ be zero on $M_m'$ and non-zero on $M_{2q-m}'$.
Then the composite $j \alpha$ is stably trivial.
One can check this fact by looking at the adjoint of $j \alpha$.
Similarly, let $i: k \to M_m' \oplus M_{2q-m}'$ be zero on $M_m'$ and non-zero on $M_{2q-m}'$.
The composite $\beta i$ is stably trivial as well.

The kernel of the non-zero map $M_{2q-m}' \to k$ is 
$M ( (ab)^{q-\frac{m+1}{2}} (ba)^{\frac{m+1}{2}-q} )$,
which we denote $K_{2q-m}$.
Let $\psi: M_m' \oplus K_{2q-m} \to M_m' \oplus M_{2q-m}'$ be the fibre of $j$, and
let $\phi = \beta \psi$.
Starting from the relation $\phi = \beta \psi$,
we form the following octahedron:
\begin{equation}\label{eq:oct1}
\begin{gathered}
\xymatrix @R=10pt @C=18pt {
                              &                                   & \Omega k \ar[dl] \ar[dr]  &                                                                                            &             \\  
                              & M \oplus \Omega k \ar[dd] \ar[rr] &                           & M_m' \oplus
 											        K_{2q-m} \ar[dd]^-{\psi}\ar[dr]^-{\phi}  &             \\   
M  \ar[ur]^-{\gamma} \ar[dr]_{(1-xy)^m\ } &                                   &                           &                                                                                            & M\\  
                              & M \ar[rr]^-{\alpha}\ar[dr]_0       &                           & M_m' \oplus M_{2q-m}' \ar[dl]^-j\ar[ur]_-{\beta}                                             &             \\  
                              &                                   &  k                        &                                                                                            &.}
\end{gathered}
\end{equation}
Since $M$ is one-periodic, we have $\Omega M = M$ at the left, and
the octahedron defines a map $\gamma: M \to M \oplus \Omega k$.
The first component of $\gamma$ is $(1-xy)^m : M \to M$ and
the second component $M \to \Omega k$ is stably non-trivial if $m > 0$.

We can use $K_{2q-m}$ to build
the module $M((ab)^{q-\frac{m}{2}-1})$, using the triangle
\[ k \llra{\tilde{\theta}} K_{2q-m} \lra
   M((ab)^{q-\frac{m}{2}-1}) \oplus M((ba)^{q-\frac{m}{2}-1}) \lra \Sigma k,\]
where $\tilde{\theta}$ is the inclusion of the socle.
Now let $\theta : k \to M_m' \oplus K_{2q-m}$ have components $0 : k \to M_m'$ and $\tilde{\theta} : k \to K_{2q-m}$.
Then $\phi \theta = \beta \psi \theta = \beta i$ is stably trivial, and
we get another octahedron based on the relation $0 = \phi \theta$:
\begin{equation*}\label{eq:oct2}
\begin{gathered}
\xymatrix @R=10pt @C=15pt{
                              &                                   &  k \ar[dl]_-{\theta} \ar[dr]^-0  &                                                                                            &             \\  
                              & M_m' \oplus 
			        K_{2q-m} \ar[dd] \ar[rr]^-{\phi} &                           & M \ar[dd] \ar[dr]^{\Sigma \gamma} &             \\   
M \oplus \Omega k \ar[ur] \ar[dr] &                                   &                           &                                                                                            & M \oplus k\\  
                              & \hspace*{-15pt}M _m' \oplus
                                M ( (ab)^{q-\frac{m}{2}-1} )  \oplus M ((ba)^{q-\frac{m}{2}-1}) \ar[rr]\ar[dr]       &                           & M \oplus \Sigma k \ar[dl]\ar[ur]^{f}                                             &             \\  
                              &                                   &  \Sigma k                        &                                                                                            &.}
\end{gathered}
\end{equation*}
The bottom horizontal triangle
\[
   M    \oplus \Omega k \lra
   M_m' \oplus M ( (ab)^{q-\frac{m}{2}-1} )  \oplus M ((ba)^{q-\frac{m}{2}-1}) \lra
   M    \oplus \Sigma k \llra{f}
   M    \oplus k
\]
shows that $\stgl(M((ab)^s))$ and $\stgl(M((ab)^{s-1}a))$ are at most $2$
for $ \frac{q-1}{2} \leq s \leq q-1$.
For example, take $m = 0$. Then the map $(1-xy)^m$ is the identity on $M$.
Picking the right basis, we see that the map $\Sigma \gamma: M \to M \oplus k$
is the identity on $M$ and zero on $k$.
It follows that the map $f$ is the direct sum of the identity map on $M$ and
a non-trivial map $\Sigma k \to k$, and
we find $M((ab)^{q-1})$ and $M((ba)^{q-1})$ as summands of the fibre of the map $\Sigma k \to k$.

To construct the modules $M((ab)^s)$ and $M((ab)^sa)$ for $s$ small,
we consider the triangle starting with $\Sigma \gamma$:
\[M \llra{\Sigma \gamma}
   M \oplus k \lra
   M_{2q-m}' \oplus M((ba)^{-\frac{m}{2}} (ab)^{\frac{m}{2}}) \llra{\Sigma \phi} M. \]
Then we can build the modules $M ((ba)^{\frac{m}{2}})$ and $M((ab)^{\frac{m}{2}})$ from $M((ba)^{-\frac{m}{2}} (ab)^{\frac{m}{2}})$
using the triangle 
\[ \Omega k \llra{\tilde{\theta'}} M((ba)^{-\frac{m}{2}} (ab)^{\frac{m}{2}}) \lra
   M ((ba)^{\frac{m}{2}}) \oplus M((ab)^{\frac{m}{2}})
  \lra k .\]
Here the rightmost map is any map that is stably non-trivial on both summands,
and $\tilde{\theta'}$ is defined by the triangle.
Let $\theta' : \Omega k \to M_{2q-m}' \oplus M((ba)^{-\frac{m}{2}} (ab)^{\frac{m}{2}})$
have components $0$ and $\tilde{\theta'}$.
One can check that $(\Sigma \phi) \theta'$ is stably trivial and that we get a triangle
similar to the one above:
\[ M \oplus \Omega k \lra
   M \oplus k \lra
   M ( (ab)^{\frac{m}{2}} ) \oplus M ((ba)^{\frac{m}{2}}) \oplus M_{2q-m}' \lra
   M \oplus k . \]
It follows that $\stgl(M((ab)^s))$ and $\stgl(M((ab)^{s-1} a))$ are at most $2$
for $1 \leq s \leq \frac{q-1}{2}$.

The two remaining cases are $M((ab)^0) = k$
and $M((ab)^{q-1} a) \iso k \up_{\smd{1,y}}^{D_{4q}}$,
both of which have strong ghost length $1$, so we are done.
\end{proof}

We illustrate the triangles in the octahedron~\eqref{eq:oct1} as follows,
taking $q=4$ and $m=2$.
The triangle~\eqref{eq:(1-g)m} corresponds to a short exact sequence
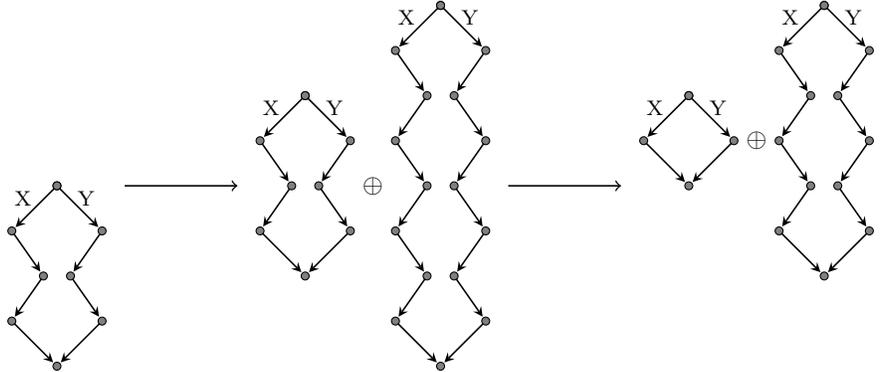
\begin{center}
\begin{tikzpicture}[scale=0.6]
\Dstartt{0,1}
\Dnext{X}{post}{(-1,-1)}
\DY
\DX
\DYw
\Dxw
\Dy
\Dx
\Dnext{Y}{pre}{(-1,1)}
\Dend 

\draw[->,semithick] (1.5,1) -- (4,1);

\Dstartt{5.5,3}
\Dnext{X}{post}{(-1,-1)}
\DY
\DX
\DYw
\Dxw
\Dy
\Dx
\Dnext{Y}{pre}{(-1,1)}
\Dend

\draw (7,1) node {$\oplus$};

\Dstartt{8.5,5}
\Dnext{X}{post}{(-1,-1)}
\DY \DX
\DY \DX
\DY \DX
\DYw \Dxw
\Dy \Dx
\Dy \Dx
\Dy \Dx
\Dnext{Y}{pre}{(-1,1)}
\Dend

\draw[->,semithick] (10,1) -- (12.5,1) ;

\Dstartt{14,3}
\Dnext{X}{post}{(-1,-1)}
\DYw
\Dxw
\Dnext{Y}{pre}{(-1,1)}
\Dend 

\draw (15.5,2) node {$\oplus$};

\Dstartt{17,5}
\Dnext{X}{post}{(-1,-1)}
\DY \DX
\DY \DX
\DYw \Dxw
\Dy \Dx
\Dy \Dx
\Dnext{Y}{pre}{(-1,1)}
\Dend
\draw (18.5,-1) node {,};
\end{tikzpicture}
\end{center}
where a free module $kD_{16}$ has been included to make the first map into an injection.
In these diagrams, the downward-left arrows indicate the action of $X = x-1$
and the downward-right arrows indicate the action of $Y = y-1$.

And the top horizontal triangle
\[ M \llra{\gamma} M \oplus \Omega k \lra M_2' \oplus M ( (ab)^{2}ab^{-1} (ba)^{-2}) \llra{\phi} M \]
in~\eqref{eq:oct1} 
has $\Omega k$ in place of the free summand and 
corresponds to a short exact sequence
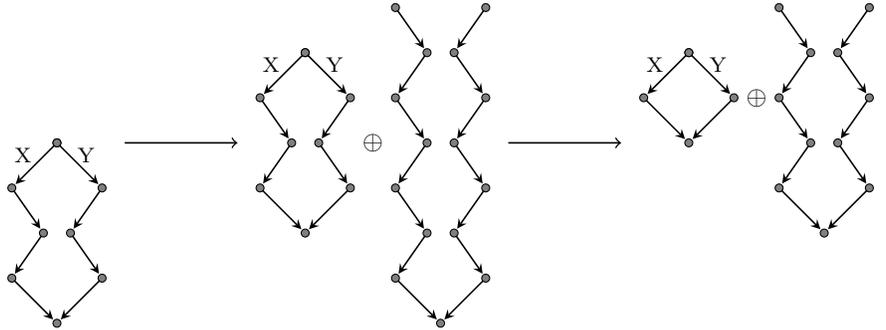
\begin{center}
\begin{tikzpicture}[scale=0.6]


\Dstartt{0,1}
\Dnext{X}{post}{(-1,-1)}
\DY
\DX
\DYw
\Dxw
\Dy
\Dx
\Dnext{Y}{pre}{(-1,1)}
\Dend 

\draw[->,semithick] (1.5,1) -- (4,1);

\Dstartt{5.5,3}
\Dnext{X}{post}{(-1,-1)}
\DY
\DX
\DYw
\Dxw
\Dy
\Dx
\Dnext{Y}{pre}{(-1,1)}
\Dend

\draw (7,1) node {$\oplus$};

\Dstartt{7.5,4}
\DY \DX
\DY \DX
\DY \DX
\DYw \Dxw
\Dy \Dx
\Dy \Dx
\Dy \Dx
\Dend

\draw[->,semithick] (10,1) -- (12.5,1);

\Dstartt{14,3}
\Dnext{X}{post}{(-1,-1)}
\DYw
\Dxw
\Dnext{Y}{pre}{(-1,1)}
\Dend 

\draw (15.5,2) node {$\oplus$};

\Dstartt{16,4}
\DY \DX
\DY \DX
\DYw \Dxw
\Dy \Dx
\Dy \Dx
\Dend
\draw (18.5,-1) node {.};


\end{tikzpicture}
\end{center}
One can check that the map $\theta: k \to M_2' \oplus M ( (ab)^{2}ab^{-1} (ba)^{-2})$
factors through the middle term,
so that $\phi \theta$ is stably trivial.


\begin{thebibliography}{bib}

\bibitem{A} 
J. L. Alperin. 
\textit{Local representation theory.}
Cambridge Studies in Advanced Mathematics \textbf{11}. 
Cambridge Univ. Press, Cambridge, 1986.

\bibitem{Benson}
D. J. Benson.
{\em Representations and cohomology I}.
Cambridge Univ. Press, Cambridge, 1998.

\bibitem{Principal} 
D. J. Benson. 
Cohomology of modules in the principal block of a finite group.
{\em New York J. Math.} \textbf{1} (1994/95), 196--205.

\bibitem{BCR} 
D. J. Benson, J.F. Carlson and G.R. Robinson. 
On the vanishing of group cohomology.
{\em J. Algebra} \textbf{131} (1990), 40--73.

\bibitem{GH for p} 
D. J. Benson, S. K. Chebolu, J. D. Christensen and J. Min\'{a}\v{c}. 
The generating hypothesis for the stable module category of a $p$-group. 
\textit{J. Algebra} \textbf{310}(1) (2007), 428--433.

\bibitem{GH general}
A. M. Bohmann and J. P. May.
A presheaf interpretation of the generalized Freyd conjecture.
{\em Theory Appl. Categ.} \textbf{26}(16) (2012), 403--411. 

\bibitem{Carlson}
J. F. Carlson.
{\em Modules and group algebras}. 
Lectures in Mathematics, ETH Z\"{u}rich. Birkh\"{a}user Verlag, Basel, 1996.
Notes by Ruedi Suter.

\bibitem{GH split}
J. F. Carlson, S. K. Chebolu and J. Min\'{a}\v{c}.
Freyd's generating hypothesis with almost split sequences.
{\em Proc. Amer. Math. Soc.} \textbf{137} (2009), 2575--2580.

\bibitem{sgh}
J. F. Carlson, S. K. Chebolu and J. Min\'{a}\v{c}.
Strong ghost maps in the stable module category.
{\em Preprint}.

\bibitem{admit}
S. K. Chebolu, J. D. Christensen and J. Min\'{a}\v{c}.
Groups which do not admit ghosts.
{\em Proc. Amer. Math. Soc.} \textbf{136}(4) (2008), 1171--1179.

\bibitem{Gh in rep}
S. K. Chebolu, J. D. Christensen and J. Min\'{a}\v{c}.
Ghosts in modular representation theory.
{\em Adv. Math.} \textbf{217}(6) (2008), 2782--2799.

\bibitem{GH per}
S. K. Chebolu, J. D. Christensen and J. Min\'{a}\v{c}.
Freyd's generating hypothesis for groups with periodic cohomology.
{\em Canad. Math. Bull.}, \textbf{55}(1) (2012), 48--59.

\bibitem{Chr}
J. D. Christensen.
Ideals in triangulated categories: phantoms, ghosts and skeleta.
{\em Adv. Math.} \textbf{136}(2) (1998), 284--339.

\bibitem{rel ha} 
J. D. Christensen and M. Hovey. 
Quillen model structures for relative homological algebra. 
\textit{Math. Proc. Cambridge Philos. Soc.} \textbf{133}(2) (2002), 261--293.

\bibitem{Gh num} 
J. D. Christensen and G. Wang. 
Ghost numbers of group algebras.
\url{http://arXiv.org/abs/1301.5740v1}

\bibitem{freydGH}
P. Freyd.
Stable homotopy.
In {\em Proc. Conf. Categorical Algebra (La Jolla, Calif., 1965)},
pages 121--172. Springer, New York, 1966.

\bibitem{B} 
B. K\"{u}lshammer. 
The principal block idempotent.
{\em Arch. Math.} \textbf{56}(4) (1991), 313--319.

\bibitem{Neeman}
A. Neeman.
The connection between the K-theory localization theorem of Thomason, Trobaugh and Yao and
the smashing subcategories of Bousfield and Ravenel.
{\em Ann. Sci. \'{E}cole Norm. Sup. (4)} \textbf{25}(5) (1992), 547--566. 

\end{thebibliography}
\end{document}